\documentclass[11pt,english, a4paper]{article}

\usepackage[english]{babel}
\usepackage[utf8x]{inputenc}
\usepackage[T1]{fontenc}
\usepackage{authblk}

\usepackage[a4paper,top=3cm,bottom=3cm,left=3cm,right=3cm,marginparwidth=1.75cm]{geometry}

\usepackage[colorlinks=true, allcolors=blue,hypertexnames=false]{hyperref}
\usepackage{amsmath}
\usepackage{empheq}
\usepackage{graphicx}
\usepackage[colorinlistoftodos]{todonotes}
\usepackage{amsmath}
\usepackage{mathtools}
\usepackage{upgreek}
\usepackage{amssymb}
\usepackage{amsfonts}
\usepackage{amsthm}
\usepackage{enumerate}
\usepackage{mathrsfs}
\usepackage{fontenc}
\usepackage{verbatim}
\usepackage{framed}
\usepackage{algorithm}
\usepackage{bbm}
\usepackage{algpseudocode}
\usepackage{appendix}
\usepackage{color}
\usepackage{cite}
\usepackage{epstopdf}
\usepackage{tabularx}

\numberwithin{equation}{section}
\theoremstyle{plain}
\newtheorem{thm}{Theorem}

\newtheorem{prop}[thm]{Proposition}

\newtheorem{assu}{Assumption}

\theoremstyle{definition}

\theoremstyle{remark}
\newtheorem*{rem}{Remark}

\newcommand{\transpose}{^{\operatorname{T}}}
\newcommand{\rmd}{\mathrm{d}}

\title{Three algorithms for solving high-dimensional fully-coupled FBSDEs through deep learning}
\author[1]{Shaolin Ji}
\author[2]{Shige Peng}
\author[1]{Ying Peng}
\author[2]{Xichuan Zhang}
\affil[1]{Shandong University-Zhongtai Securities Institute for Financial Studies, Shandong University, 250100, China}
\affil[2]{School of Mathematics, Shandong University, 250100, China}

\begin{document}
\maketitle

\begin{abstract}
Recently, the deep learning method has been used for solving forward-backward stochastic differential equations (FBSDEs) and parabolic partial differential equations (PDEs). It has good accuracy and performance for high-dimensional problems. In this paper, we mainly solve fully coupled FBSDEs through deep learning and provide three algorithms. Several numerical results show remarkable performance especially for high-dimensional cases.
\\ \\
\textbf{Keywords} deep learning $\cdot$ fully-coupled FBSDEs $\cdot$ high-dimensional equation $\cdot$ stochastic control
\end{abstract}

\section{Introduction}\label{sec:intro}
In 1990, Pardoux and Peng proved the existence and uniqueness of the adapted solution for nonlinear BSDEs~\cite{pengBSDE1990} and then found important applications of BSDEs in finance. When a BSDE is coupled with a (forward) stochastic differential equation (SDE), the system is usually called a forward-backward stochastic differential equation (FBSDE). The FBSDE is an important tool for a wide range of application areas. For example, it can be used to solve the stochastic optimization problems, and meanwhile the stochastic optimization provides a very important way for solving many AI applications such as robotics and aerospace. In financial fields, when constructing a derivatives pricing model, it is necessary to consider the credit value adjustment (CVA) and the funding value adjustment (FVA) in many cases. However, the computation of CVA and FVA greatly increase the difficulty for financial derivatives pricing. In this instance, the FBSDE appears as an effective tool for the modeling of this kind of complex derivatives pricing.

In most situations, it is impossible to obtain an explicit solution of a FBSDE. Therefore, it is necessary to find the approximate solution. In this paper, we aim to obtain the numerical solution of the following fully-coupled FBSDE through deep learning:

\begin{equation}\label{eq:FBSDE}
\begin{cases}
X_{t} = X_{0} + \displaystyle\int_{0}^{t}b(s,X_{s},Y_{s},Z_{s})\mathrm{d}s + \int_{0}^{t}\sigma(s,X_{s},Y_{s},Z_{s}) \mathrm{d}W_{s},  \\
Y_{t} = g(X_{T}) + \displaystyle\int_{t}^{T}f(s,X_{s},Y_{s},Z_{s})\, \mathrm{d}s - \int_{t}^{T} Z_{s}\mathrm{d}W_{s}.
\end{cases}
\end{equation}

There are several ways to find the numerical solution of FBSDE \eqref{eq:FBSDE}. Based on the relationship between FBSDEs and PDEs (see \cite{Feynman_for_FBSDE}), numerical methods for solving the PDEs, such as the finite element method, the finite difference method, or the sparse grid method, can be applied to solve the FBSDEs. Ma et al studied the solvability of coupled FBSDEs and proposed a four-step approach \cite{Yong2007Forward}. Moreover, some probabilistic methods, which approximate the conditional expectation with numerical schemes, were developed to solve the FBSDEs. For example, \cite{Zhao2016Multistep} proposed a theta-scheme numerical method with high accuracy for coupled Markovian FBSDEs. \cite{time_discretize_FBSDE_Zhang} introduced a numerical scheme for coupled FBSDEs when the forward process $ X_t $ does not depend on $ Z_t $.

As is known, there is a significant difficulty for solving high dimensional BSDEs and FBSDEs, namely "curse of dimensionality". The computational complexity grows exponentially when the dimension increases, while the accuracy declines sharply. Therefore most of the aforementioned numerical methods can not deal with high-dimensional problems.

Recently, deep-learning method has achieved great success in many application areas~\cite{BengioDL}, such as computer vision, natural language processing, gaming, etc. It provides a new point of view to approximate functions and shows optimistic performance in solving problems with high-dimension features. This poses a possible way to solve the "curse of dimensionality" although the reason why deep-learning has so remarkable performance has not been proven completely.

In a recent breakthrough paper of E et al~\cite{WeinanDLforBSDE}, they solved the BSDEs from a control perspective by regarding the $Z$ term as a control. A neural network was constructed to solve high dimensional BSDEs and related PDEs. Their method has shown superior performance and accuracy on comparing with the traditional numerical methods. Lately, Han and Long~\cite{deeplearning_FBSDE} extended this method to solve FBSDEs where $ b $ and $\sigma$ in the forward SDE do not depend on the $ Z $ term and gave a posteriori error estimation.

In this paper, we solve fully-coupled FBSDEs \eqref{eq:FBSDE} by the above optimal control approach through deep neural network. Different from \cite{WeinanDLforBSDE,deeplearning_FBSDE}, we systematically explore the dependence of the term $Z$ on state precesses $X,Y$ and even $Z$ itself (in the following Algorithm 3) and propose three algorithms corresponding to different kinds of state feedback. In order to do this, we should design different cost functionals which make it possible to solve FBSDEs \eqref{eq:FBSDE} by the optimal control approach. In the first algorithm (Algorithm 1), we adopt the same state feedback form as that in Han and Long \cite{deeplearning_FBSDE}: the control $ \tilde{Z}_t $ is supposed to be dependent on the states $ \tilde{X}_t $ of the forward SDE and $ \tilde{Y}_t $ of the BSDE. We generalize it to solve \eqref{eq:FBSDE} in which $b$ and $\sigma$ depend on the $ Z $ term. In the second algorithm (Algorithm 2), we take the state $\tilde{X}_t$ as the feedback. Besides the control $ \tilde{Z}_t $, $ \tilde{Y}_t $ in the forward SDE should be regarded as a new control and denoted as $ u_t $. Both the controls $ u_t $ and $ \tilde{Z}_t $ are supposed to be dependent on the state $ \tilde{X}_t $ of the forward SDE. The price of doing this is that we must change the form of the cost function in Algorithm 1. A new penalty term is added to the cost function to punish the difference between the control $ u_t $ and the solution $ \tilde{Y}_t $ of the backward SDE. The third algorithm (Algorithm 3) is inspired by the idea of the Picard iteration (see \cite{Feynman_for_FBSDE}). Given the initial pathes $(\tilde{Y}^0,\tilde{Z}^0)$, the next iteration pathes $(\tilde{X}^{k+1},\tilde{Y}^{k+1},\tilde{Z}^{k+1})$ are dependent on the current pathes $(\tilde{Y}^{k},\tilde{Z}^{k})$ in Picard iteration. $ \tilde{Z}^{k+1} $ is supposed to be dependent on $(\tilde{X}^{k+1},\tilde{Y}^{k},\tilde{Z}^{k})$. Different from the state variable feedback in Algorithm 1 and 2, Algorithm 3 is a path-dependent iteration one which has potential applications in the calculation of FBSDE \eqref{eq:FBSDE} with random coefficients.

These three methods for solving FBSDE \eqref{eq:FBSDE} can also be widely applied to solving high-dimensional BSDEs. It is well-known that Feynman-Kac formula gives the probabilistic interpretation of the solution of linear PDEs. \cite{Peng1991Probabilistic} and related literatures obtained generalized Feynman-Kac formulas which establish the relationship between FBSDEs and nonlinear PDEs. The corresponding PDE of FBSDE \eqref{eq:FBSDE} is
\begin{equation*}
	\begin{cases}
		-\rmd u = \Big\{\sum_{i,j}\dfrac{1}{2}[\sigma\sigma\transpose(t,x,u,v)]_{i,j}\dfrac{\partial^2 u}{\partial x_i\partial x_j} + \dfrac{\partial u}{\partial x}b(t,x,u,v) \\
		\displaystyle\qquad + \sum_i\nabla v_i\sigma_i\transpose(t,x,u,v) + f(t,x,u,v) \Big\}\rmd t - v\rmd W_t,   \\
		v(t,x)=\nabla u\sigma(t,x,u,v), \qquad x\in \mathbb{R}^n,\\
	\end{cases}
\end{equation*}
with terminal condition $u(T,x) = g(x)$.

Consequently, from the numerical results, all the three algorithms can approximate the solution of the FBSDE \eqref{eq:FBSDE} and perform well in high-dimensional cases. As shown in the examples in Section 5, the relative errors of these algorithms are less than 1\%.

The remainder of this paper is organized as following. In Section 2, we firstly present some preliminaries on FBSDEs and in particular, give the existence and uniqueness conditions of fully-coupled FBSDEs. In Section 3, the relationship between FBSDEs and an optimal control problem are presented, which indicates that the FBSDEs can be solved from a control perspective. According to different kinds of state feedback, we propose three optimal control methods for solving FBSDE \eqref{eq:FBSDE}. In Section 4, we present our numerical schemes and the corresponding algorithms and show the neural network architecture. Section 5 gives some examples and shows the comparison among different algorithms for solving coupled FBSDEs.

\section{Preliminaries on FBSDEs}\label{sec:Preliminary}

In this section, we mainly introduce the form of FBSDEs and the existence and uniqueness conditions of fully-coupled FBSDEs \cite{Peng1999Fully}.

Let $ T>0 $ and $ (\Omega,\mathcal{F},\mathbb{P}, \mathbb{F}) $ be a filtered probability space, where $ W:[0,T] \times \Omega \rightarrow \mathbb{R} $ is a $ d $-dimensional standard Brownian motion on $ (\Omega,\mathcal{F},\mathbb{P}) $, $ \mathbb{F}=\{\mathcal{F}_{t}\}_{0\leq t\leq T} $ is the natural filtration generated by the Brownian motion $ W $. $X_{0}\in \mathcal{F}_{0}$ is the initial condition for the FBSDE.

Considering the fully-coupled FBSDE \eqref{eq:FBSDE}, where $ \{X_{t}\}_{0\leq t\leq T},\{Y_{t}\}_{0\leq t\leq T},\{Z_{t}\}_{0\leq t\leq T} $ are $ \mathbb{F} $-adapted stochastic processes taking value in $ \mathbb{R}^{n}, \mathbb{R}^{m}, \mathbb{R}^{m\times d} $, respectively. The functions
\begin{align*}
b: &\Omega\times [0,T]\times \mathbb{R}^{n}\times \mathbb{R}^{m}\times \mathbb{R}^{m\times d} \rightarrow \mathbb{R}^{n}\\
\sigma: &\Omega\times [0,T]\times \mathbb{R}^{n}\times \mathbb{R}^{m}\times \mathbb{R}^{m\times d} \rightarrow \mathbb{R}^{n\times d}\\
f: &\Omega\times [0,T]\times \mathbb{R}^{n}\times \mathbb{R}^{m}\times \mathbb{R}^{m\times d} \rightarrow \mathbb{R}^{m}\\
g: &\Omega\times [0,T]\times \mathbb{R}^{n} \rightarrow \mathbb{R}^{m}
\end{align*}
are deterministic globally continuous functions. $ b $ and $ \sigma $ are the \textit{drift} coefficient and \textit{diffusion} coefficient of $ X $ respectively, and $ f $ is referred to as the \textit{generator} of the coupled FBSDE. If there is a triple $ (X_{t}, Y_{t}, Z_{t}) $ satisfies the above FBSDE on $ [0, T] $, $ \mathbb{P} $-almost surely, square integrable and $ \mathcal{F}_{t} $-adapted, the triple $ (X_{t}, Y_{t}, Z_{t}) $ are called the solutions of FBSDE (\ref{eq:FBSDE}). When functions $ b $ and $\sigma$ are independent of both $ Y $ and $ Z $, FBSDE (\ref{eq:FBSDE}) is called a decoupled FBSDE.

It is well known that in BSDE theory, even if all coefficients satisfy Lipschitz condition, fully coupled FBSDE does not necessarily have solutions. So we have to give some more assumptions.

Given a $ m\times n $ full-rank matrix $ G $, we define
\begin{empheq}{align*}
u = \left( \begin{array}{c} x \\ y \\z \end{array} \right), \ A(t, u)=\left(\begin{array}{c}
-G^{T}f \\Gb\\ G\sigma
\end{array}\right)(t,u),
\end{empheq}
where $ G\sigma=(G\sigma_{1}\cdots G\sigma_{d}) $.

Firstly, we give two assumptions as the following,

\begin{assu} \label{assu:1}
        \begin{enumerate}[(i)]
          \item $ A(t,u) $ is uniformly Lipschitz with respect to $ u $;
		  \item $ A(\cdot,u) $ is in $ M^{2}(0,T) $, $ \forall u $;
		  \item $ g(x) $ is uniformly Lipschitz with respect to $ x \in \mathbb{R}^{n} $;
		  \item $ g(x) $ is in $ L^{2}(\Omega,\mathcal{F}_{T},\mathbb{P}) $, $ \forall x $.
        \end{enumerate}
\end{assu}
\begin{assu} \label{assu:2}
		\begin{align*}
	    \left\langle A(t,u)-A(t,\bar{u}), u-\bar{u} \right\rangle & \leq - \beta_{1}|G\hat{x}|^{2} - \beta_{2}(|G^{T}\hat{y}|^{2}+|G^{T}\hat{z}|^{}),\\
	    \left\langle g(x)-g(\bar{x}), G(x-\bar{x}) \right\rangle & \geq \mu_{1}|G\hat{x}|^{2},
	    \end{align*}
	    $$ \forall u =(x,y,z), \bar{u}=(\bar{x},\bar{y},\bar{z}), \hat{x} = x-\bar{x}, \hat{y} = y-\bar{y}, \hat{z} = z-\bar{z}, $$
	    where $ \beta_{1},\beta_{2} $ and $ \mu_{1} $ are given nonnegative constants with $ \beta_{1}+\beta_{2}>0,\mu_{1}+\beta_{2}>0. $
\end{assu}

Then, the existence and uniqueness theorem of FBSDEs was obtained in \cite{Peng1999Fully} and \cite{Hu1995Solution}.

\begin{thm}\label{thm:1}
	Let Assumptions~\ref{assu:1} and \ref{assu:2} hold, then there exists a unique adapted solution $ (X,Y,Z) $ of FBSDE \eqref{eq:FBSDE}.
\end{thm}

Readers are referred to Theorem 2.6 of \cite{Peng1999Fully} for the proof in detail.

For convenience, in this article, we assume that $ L $ is the Lipschitz constant satisfying Assumption \ref{assu:1} , that is $ \forall x,x',y,y',z,z' $
\begin{align*}
|l(t,x,y,z)-l(t,x',y',z')|  \leq &L(|x-x'|+|y-y'|+|z-z'|),\\
|g(x)-g(x')| \leq & L(|x-x'|),
\end{align*}
where $ l $ represents one of the functions among $ f,b $ and $\sigma$.

\section{Solving FBSDEs from an optimal control perspective}\label{sec:contr_persp}

Essentially, a deep neural network considered in this paper can be regarded as a control system, which is used to approximate the mapping from the input set to the label set. The parameters in the network can be seen as the control, and the cost function can be seen as the optimization objective. Thus, we firstly transform the problem of solving FBSDE into three optimal control problems with different kinds of feedback control.

\subsection{Case 1: Feedback control based on \texorpdfstring{$ (X,Y) $}{}}\label{subsec:3.1}

In the first case, we extend the method in \cite{WeinanDLforBSDE,deeplearning_FBSDE} for solving FBSDE \eqref{eq:FBSDE}. Consider the following variational problem:
\begin{align}
\inf_{Y_{0},\{Z_{t}\}_{0\leq t \leq T}} \mathbb{E}\left[|g(X_{T}^{Y_{0},Z}) - Y_{T}^{Y_{0},Z}|^{2}\right], \label{eq:objectiveFBSDE_1}
\end{align}
\begin{align*}
s.t.~~ X_{t}^{Y_{0},Z} = X_{0} & +
\int_{0}^{t}b(s,X_{s}^{Y_{0},Z},Y_{s}^{Y_{0},Z},Z_{s})\, \mathrm{d}s
 +  \int_{0}^{t}\sigma(s,X_{s}^{Y_{0},Z},Y_{s}^{Y_{0},Z},Z_{s})\, \mathrm{d}W_{s},
\notag \\
\qquad Y_{t}^{Y_{0},Z} = Y_{0}  - &
\int_{0}^{t}f(s,X_{s}^{Y_{0},Z},Y_{s}^{Y_{0},Z},Z_{s})\, \mathrm{d}s + \int_{0}^{t} Z_{s} \mathrm{d}W_{s},
\notag
\end{align*}
where $ Y_{0} $ is $ \mathcal{F}_{0} $-measurable random variable valued in $ \mathbb{R}^{m} $ and $ Z_{t} $ is a $ \mathcal{F}_t $-adapted and square-integrable process. The couple $ (Y_0, \{Z_t\}_{0\leq t\leq T}) $ is regarded as the control of variational problem \eqref{eq:objectiveFBSDE_1} and $ Z(\cdot) $ is a feedback control based on $ ( X_t^{Y_{0},Z}, Y_t^{Y_{0},Z}) $.

\begin{prop}\label{pro:1}
    Let Assumption \ref{assu:1} and \ref{assu:2} hold. Then the optimal control problem \eqref{eq:objectiveFBSDE_1} satisfies
	\begin{align}\label{condition:alg2}
	\inf_{Y_{0},\{Z_{t}\}_{0\leq t \leq T}} \mathbb{E}\left[|g(X_{T}^{Y_{0},Z}) - Y_{T}^{Y_{0},Z}|^{2}\right]=0,
	\end{align}
	and can be achieved. Moreover, the corresponding triple $ (X_{t}^{Y_{0},Z},Y_{t}^{Y_{0},Z},Z_{t}) $ is the unique solution of FBSDE \eqref{eq:FBSDE}.
\end{prop}

\begin{proof}
    Thanks to Assumption \ref{assu:1} and \ref{assu:2} hold, the FBSDE \eqref{eq:FBSDE} has a unique solution $ (X_{t},Y_{t},Z_{t}) $. Regarding $(Y_0,Z_t)$ as the control of the variational problem \eqref{eq:objectiveFBSDE_1}, we have
    \begin{align*}
	\inf_{Y_{0},\{Z_{t}\}_{0\leq t \leq T}} \mathbb{E}\left[|g(X_{T}^{Y_{0},Z}) - Y_{T}^{Y_{0},Z}|^{2}\right]=0,
	\end{align*}
    and then the optimal control can be achieved. The corresponding triple $ (X_{t}^{Y_{0},Z},Y_{t}^{Y_{0},Z},Z_{t}) $ is the solution of FBSDE \eqref{eq:FBSDE}. Because of Assumption \ref{assu:1} and \ref{assu:2} hold, the solution is unique.
\end{proof}

The detailed iteration algorithm will be given in Section \ref{sec:numerical}.

\subsection{Case 2: Feedback controls based on \texorpdfstring{$ X $}{}}\label{subsec:3.2}

In \eqref{eq:objectiveFBSDE_1}, the initial value of process $ Y $ is regarded as a control. Now we regard the whole process $ Y $ in the forward SDE as a control, then we have the following control problem
\begin{align}
\inf_{\{u_{t}, Z_{t}\}_{0\leq t \leq T}} \mathbb{E}\left[|g(X_{T}^{u,Z}) - Y_{T}^{u,Z}|^{2}+
\int_{0}^{T}|Y_t^{u,Z}-u_t|^2\mathrm{d}t
\right], \label{eq:objectiveFBSDE_3}
\end{align}
\begin{align*}
s.t.~~ X_{t}^{u,Z} = X_{0} & +
\int_{0}^{t}b(s,X_{s}^{u,Z},u_s,Z_{s})\, \mathrm{d}s
 +  \int_{0}^{t}\sigma(s,X_{s}^{u,Z},u_s,Z_{s})\, \mathrm{d}W_{s}, \notag \\
\qquad Y_{t}^{u,Z} = u_{0}  - & \int_{0}^{t}f(s,X_{s}^{u,Z},Y_{s}^{u,Z},Z_{s})\, \mathrm{d}s + \int_{0}^{t} Z_{s} \mathrm{d}W_{s}, \notag
\end{align*}
where $ u_{t} $ and $ Z_{t} $ are two $ \mathcal{F}_t $-adapted and square-integrable processes. The $ ( \{u_t,Z_t\})_{0\leq t\leq T} $ are the controls of the variational problem \eqref{eq:objectiveFBSDE_3}. Both $ u(\cdot),Z(\cdot) $ are feedback controls based on $ X_t^{u,Z} $.

\begin{prop}\label{pro:2}
    Let Assumption \ref{assu:1} and \ref{assu:2} hold and $ u_{t} $ is an $ \mathcal{F}_t $-adapted and square-integrable process. Then the optimal control problem \eqref{eq:objectiveFBSDE_3} satisfies
	\begin{align}\label{condition:alg3}
	\inf_{\{u_{t}, Z_{t}\}_{0\leq t \leq T}} \mathbb{E}\left[|g(X_{T}^{u,Z}) - Y_{T}^{u,Z}|^{2}+
	 \int_{0}^{T}|Y_t^{u,Z}-u_t|^2\mathrm{d}t \right]=0,
	\end{align}
	and can be achieved. Moreover, the corresponding triple $ (X_{t}^{u,Z},Y_{t}^{u,Z},Z_{t}) $ is the unique solution of FBSDE \eqref{eq:FBSDE}.
\end{prop}

\begin{proof}
    Thanks to Assumption \ref{assu:1} and \ref{assu:2} hold, the FBSDE \eqref{eq:FBSDE} has a unique solution $ (X_{t},Y_{t},Z_{t}) $. Regarding $(u_t,Z_t)$ as the control of the variational problem \eqref{eq:objectiveFBSDE_3}, we have
    \begin{align*}
	\inf_{\{u_{t}, Z_{t}\}_{0\leq t \leq T}} \mathbb{E}\left[|g(X_{T}^{u,Z}) - Y_{T}^{u,Z}|^{2}+
	 \int_{0}^{T}|Y_t^{u,Z}-u_t|^2\mathrm{d}t \right]=0,
	\end{align*}
    and then the corresponding triple $ (X_{t}^{u,Z},Y_{t}^{u,Z},Z_{t}) $ is the solution of FBSDE \eqref{eq:FBSDE}. Because of Assumption \ref{assu:1} and \ref{assu:2} hold, the solution is unique.
\end{proof}

\subsection{Case 3: Feedback control based on \texorpdfstring{$ (X,Y,Z) $}{} with Picard iteration method}\label{subsec:3.3}

Before introducing the Picard iteration with feedback $ (X,Y,Z) $, we first give an assumption to make sure the Picard iteration is convergent \cite{Feynman_for_FBSDE}.
\begin{assu}\label{assu:3}
	\begin{enumerate}[(i)]
		\item There exist $ \lambda_1,\lambda_2\in\mathbb{R} $ such that for all $t,x,x_1,x_2,y,y_1,y_2,z$ and a.s.,
		\begin{align*}
		\left\langle b(t,x_1,y,z)-b(t,x_2,y,z), x_1-x_2 \right\rangle \leq \lambda_1|x_1-x_2|^2,\\
		\left\langle f(t,x,y_1,z)-f(t,x,y_2,z), y_1-y_2 \right\rangle \leq \lambda_2|y_1-y_2|^2;
		\end{align*}
		\item There exist $ k,k_i>0,i=1,2,3,4 $, such that for all $t,x,x_1,x_2,y,y_1,y_2,z,z_1,z_2$ and a.s.,
		\begin{align*}
		&|b(t,x,y_1,z_1)-b(t,x,y_2,z_2)| \leq k_1|y_1-y_2|+k_2\|z_1-z_2\|,\\
		&|b(t,x,y,z)|\leq |b(t,0,y,z)|+k(1+|x|),\\
		&|f(t,x_1,y,z_1)-f(t,x_2,y,z_2)| \leq k_3|x_1-x_2|+k_4\|z_1-z_2\|,\\
		&|f(t,x,y,z)|\leq |f(t,x,0,z)|+k(1+|y|);
		\end{align*}
		\item There exist $ k_i,i=5,6,7 $, such that for all $t,x_1,x_2,y_1,y_2,z_1,z_2$ and a.s.,
		\begin{align*}
		&\|\sigma(t,x_1,y_1,z_1)-\sigma(t,x_2,y_2,z_2)\|^2\\
		\leq & k_5^2|x_1-x_2|^2 + k_6^2|y_1-y_2|^2 + k_7^2\|z_1-z_2\|^2;
		\end{align*}
		\item There exists $ k_8 $, such that for all $x_1,x_2$ and a.s.,
		\begin{align*}
		|g(x_1) - g(x_2)|\leq k_8|x_1-x_2|;
		\end{align*}
		\item The processes $ b(\cdot,x,y,z), \sigma(\cdot,x,y,z)$ and $f(\cdot,x,y,z) $ are $ \mathcal{F}_t $-adapted, and the random variable $ g(x) $ is $ \mathcal{F}_T $-measurable, for all $ (x,y,z) $. Moreover, the following holds:
		\begin{align*}
		\mathbb{E}\int_{0}^{T}|b(s,0,0,0)|^2\mathrm{d}s + \mathbb{E}\int_{0}^{T}\|\sigma(s,0,0,0)\|^2\mathrm{d}s\\
		+ \mathbb{E}\int_{0}^{T}|f(s,0,0,0)|^2\mathrm{d}s + \mathbb{E}|h(0)|^2 <\infty.
		\end{align*}
	\end{enumerate}
\end{assu}
The notations $ |\cdot| $ and $ \|\cdot\| $ are denoted as the square-root of the sum of squares of the components of a vector and a matrix, respectively. Under Assumption 3 and $\lambda_1+\lambda_2<-(k_4^2+k_5^2)^2/2$, FBSDE \eqref{eq:FBSDE} has a unique solution $ (X_{t}, Y_{t}, Z_{t}) $. Pardoux and Tang's results \cite{Feynman_for_FBSDE} have shown that \eqref{eq:FBSDE} can be constructed via Picard iteration
\begin{equation}\label{Picard_FBSDE}
\begin{cases}
X_{t}^{k+1} = X_{0} +  \displaystyle\int_{0}^{t}b(s,X_{s}^{k+1},Y_{s}^{k},Z_{s}^{k})\mathrm{d}s + \int_{0}^{t}\sigma(s,X_{s}^{k+1},Y_{s}^{k},Z_{s}^{k}) \mathrm{d}W_{s},  \\
Y_{t}^{k+1} = g(X_{T}^{k+1}) +  \displaystyle\int_{t}^{T}f(s,X_{s}^{k+1},Y_{s}^{k+1},Z_{s}^{k+1}) \mathrm{d}s - \int_{t}^{T} Z_{s}^{k+1}\mathrm{d}W_{s},
\end{cases}
\end{equation}
when $ Y^{0}, Z^{0} $ are given, $ k $ is denoted as the iteration step. Therefore, the solution $ (X_{t}^{k+1}, Y_{t}^{k+1}, Z_{t}^{k+1}) $ of the decoupled FBSDE \eqref{Picard_FBSDE} converges to to the solution $ (X_{t}, Y_{t}, Z_{t}) $ of \eqref{eq:FBSDE} when $ k $ tends to infinity, i.e.
\begin{equation}\label{Tang}
\lim\limits_{k\rightarrow \infty}\mathbb{E}\left[ \sup\limits_{0\leq t\leq T}(\Delta X_{t,k+1}^2
+ \Delta Y_{t,k+1}^2) + \int_{0}^{T}\Delta Z_{t,k+1}^2\, \mathrm{d}t \right] =0,
\end{equation}
where $\Delta X_{t,k+1} = |X_{t}^{k+1}-X_{t}|, \Delta Y_{t,k+1} =  |Y_{t}^{k+1}-Y_{t}|$, $\Delta Z_{t,k+1} = |Z_{t}^{k+1}-Z_{t}|$.

Regarding $ \tilde{Y}_{0}^{k+1} $ and $ \{\tilde{Z}_{t}^{k+1}\}_{0\leq t\leq T} $ as controls, we consider the following control problem
\begin{align}
\inf_{\tilde{Y}_0^{k+1},\{\tilde{Z}_t^{k+1}\}_{0\leq t \leq T}} \mathbb{E}\left[|g(\tilde{X}_{T}^{k+1}) - \tilde{Y}_{T}^{k+1}|^{2}\right], \label{eq:control1}
\end{align}
\begin{align}
s.t.~~ \tilde{X}_{t}^{k+1}= &X_{0} + \int_{0}^{t}b(s,\tilde{X}_{s}^{k+1},\tilde{Y}_{s}^{k},\tilde{Z}_{s}^{k})\mathrm{d}s + \int_{0}^{t}\sigma(s,\tilde{X}_{s}^{k+1},\tilde{Y}_{s}^{k},\tilde{Z}_{s}^{k}) \mathrm{d}W_{s} , \label{Picard_SDE} \\
\qquad \tilde{Y}_{t}^{k+1}= &\tilde{Y}_{0}^{k+1} - \int_{0}^{t}f(s,\tilde{X}_{s}^{k+1},\tilde{Y}_{s}^{k+1},\tilde{Z}_{s}^{k+1}) \mathrm{d}s + \int_{0}^{t} \tilde{Z}_{s}^{k+1}\mathrm{d}W_{s}, \notag
\end{align}
where $\tilde{Y}_0^{k+1}$ and $\tilde{Z}_t^{k+1}$ are taking valued in $\mathbb{R}$ and $M^2(0,T)$, respectively. The couple $ (\tilde{Y}_0^{k+1}, \{\tilde{Z}_t^{k+1}\}_{0\leq t\leq T}) $ is regarded as the controls and $ Z(\cdot) $ is a feedback control based on $ (\tilde{X}_t^{k+1},\tilde{Y}_t^{k},\tilde{Z}_t^{k}) $. In the following, we will show that control problem \eqref{eq:control1} is equivalent to FBSDE \eqref{Picard_FBSDE}.

When $ \tilde{Y}^{k}, \tilde{Z}^{k} $ are known, we consider the SDEs \eqref{Picard_SDE} in \eqref{eq:control1}, which  has infinite number of solutions because both the initial value $ \tilde{Y}_{0}^{k+1} $ and the process $ \{\tilde{Z}_{t}^{k+1}\}_{0\leq t\leq T} $ are uncertain. Given $ \tilde{Y}_{0}^{k+1} $ and the process $ \{\tilde{Z}_{t}^{k+1}\}_{0\leq t\leq T} $, \eqref{Picard_SDE} is a system of forward stochastic differential equations with initial condition $ (X_0, \tilde{Y}_{0}^{k+1}) $. Under Assumption \ref{assu:3} and $\lambda_1+\lambda_2<-(k_4^2+k_5^2)^2/2$, the equations \eqref{Picard_SDE} has a unique solution $ (\tilde{X}_{t}^{k+1},\tilde{Y}_{t}^{k+1}) $ \cite{oksendal1985stochastic} determined by $ \tilde{Y}_{0}^{k+1} $ and $ \{\tilde{Z}_{t}^{k+1}\}_{0\leq t\leq T} $.

In the following Theorem \ref{thm:main}, we will show that the solution $ (\tilde{X}_t^{k+1},\tilde{Y}_t^{k+1},\tilde{Z}_t^{k+1}) $ of \eqref{Picard_SDE} converges to the solution $ (X_t,Y_t,Z_t) $ of FBSDE \eqref{eq:FBSDE} when $ \mathbb{E}\left[|g(\tilde{X}_{T}^{k+1}) - \tilde{Y}_{T}^{k+1}|^{2}\right] $ goes to zero as $ k $ tends to infinity, which means that solving the control problem \eqref{eq:control1} is equivalent to solving the FBSDE \eqref{eq:FBSDE}.

\begin{thm}\label{thm:main}
	Suppose Assumption \ref{assu:3} holds true, $\lambda_1+\lambda_2<-(k_4^2+k_5^2)^2/2$, and there exist $ C_5 \leq 1 $, where $C_5$ is constant dependent on $L, T$ and $L$ is the upper bound of the coefficients in Assumption \ref{assu:3}. If $ \mathbb{E}\left[|g(\tilde{X}_{T}^{k+1}) - \tilde{Y}_{T}^{k+1}|^{2}\right] $ satisfies
	\begin{align*}
	\lim\limits_{k\rightarrow \infty}\mathbb{E}\left[|g(\tilde{X}_{T}^{k+1}) - \tilde{Y}_{T}^{k+1}|^{2}\right] = 0,
	\end{align*}
	for small enough $T$, then the solution of SDE \eqref{Picard_SDE} $ (\tilde{X}_{t}^{k+1}, \tilde{Y}_{t}^{k+1}, \tilde{Z}_{t}^{k+1}) $ satisfies
	\begin{align}
	\lim\limits_{k\rightarrow \infty}\mathbb{E}\left[\sup\limits_{0\leq t\leq T}(|\tilde{X}_{t}^{k+1}-X_{t}|^2+|\tilde{Y}_{t}^{k+1}-Y_{t}|^2) +\int_{0}^{T}|\tilde{Z}_{t}^{k+1}-Z_{t}|^2\mathrm{d}t\right]=0 \label{SDEconditio}
	\end{align}
	for given $(Y^0, Z^0)$, where $ (X_{t}, Y_{t}, Z_{t}) $ is the solution of FBSDE \eqref{eq:FBSDE}.
\end{thm}

\begin{proof}
    The proof of this theorem is divided into two steps.
	\begin{itemize}
		\item[Step 1:] Supposing the following equation
		\begin{equation}\label{Picard_FBSDEHat}
		\begin{cases}
		\hat{X}_{t}^{k+1} = X_{0} + \displaystyle \int_{0}^{t}b(s,\hat{X}_{s}^{k+1},\tilde{Y}_{s}^{k},\tilde{Z}_{s}^{k})\mathrm{d}s + \int_{0}^{t}\sigma(s,\hat{X}_{s}^{k+1},\hat{Y}_{s}^{k},\hat{Z}_{s}^{k}) \mathrm{d}W_{s},  \\
		\hat{Y}_{t}^{k+1} = g(\hat{X}_{T}^{k+1}) + \displaystyle \int_{t}^{T}f(s,\hat{X}_{s}^{k+1},\hat{Y}_{s}^{k+1},\hat{Z}_{s}^{k+1}) \mathrm{d}s - \int_{t}^{T} \hat{Z}_{s}^{k+1}\mathrm{d}W_{s},
		\end{cases}
		\end{equation}
		has a solution $ (\hat{X}_t^{k+1},\hat{Y}_t^{k+1},\hat{Z}_t^{k+1}) $. Let
			\begin{align*}
			\delta X_t^{k+1} &= \hat{X}_t^{k+1} - X_t^{k+1}, \\
			\delta Y_t^{k+1} &= \hat{Y}_t^{k+1} - Y_t^{k+1}, \\
			\delta Z_t^{k+1} &= \hat{Z}_t^{k+1} - Z_t^{k+1}, \\
			\delta Y_T^{k+1} &= g(\hat{X}_T^{k+1})-g(X_T^{k+1}), \\
			\delta b_t &= b(t,\hat{X}_{t}^{k+1},\tilde{Y}_{t}^{k},\tilde{Z}_{t}^{k}) - b(t,X_{t}^{k+1},Y_{t}^{k},Z_{t}^{k}), \\
			\delta \sigma_t &= \sigma(t,\hat{X}_{t}^{k+1},\tilde{Y}_{t}^{k},\tilde{Z}_{t}^{k}) - \sigma(t,X_{t}^{k+1},Y_{t}^{k},Z_{t}^{k}), \\
			\delta f_t &= f(t,\hat{X}_{t}^{k+1},\hat{Y}_{t}^{k+1},\hat{Z}_{t}^{k+1}) - f(t,X_{t}^{k+1},Y_{t}^{k+1},Z_{t}^{k+1}).
			\end{align*}
			
			From \eqref{Picard_FBSDE} and \eqref{Picard_FBSDEHat}, we get
			\begin{align*}
			\delta X_t^{k+1} &= \displaystyle \int_{0}^{t}\delta b_s \mathrm{d}s - \int_{0}^{t} \delta \sigma_s\mathrm{d}W_{s},\\
			\delta Y_t^{k+1} &= \delta Y_T^{k+1} + \displaystyle \int_{t}^{T}\delta f_s \mathrm{d}s - \int_{t}^{T} \delta Z_s^{k+1}\mathrm{d}W_{s},
			\end{align*}
			whose differential form is
			\begin{align*}
			\mathrm{d}\delta X_t^{k+1} &= \delta b_t \mathrm{d}t + \delta \sigma_t\mathrm{d}W_{t},\\
			-\mathrm{d}\delta Y_t^{k+1} &= \delta f_t \mathrm{d}t - \delta Z_t^{k+1}\mathrm{d}W_{t}.
			\end{align*}
			Plugging Ito's formula into $ |\delta X_t^{k+1}|^2 $,
			\begin{align*}
			\mathrm{d}|\delta X_t^{k+1}|^2 &= 2\delta X_t^{k+1}\cdot\mathrm{d}\delta X_t^{k+1}+\mathrm{d}\delta X_t^{k+1}\cdot\mathrm{d}\delta X_t^{k+1}\\
			&= 2\delta X_t^{k+1}(\delta b_t \mathrm{d}t + \delta \sigma_t\mathrm{d}W_{t})-|\delta \sigma_t|^2\mathrm{d}t,
			\end{align*}
			integrating from $ 0 $ to $ t $,
			\begin{align*}
			|\delta X_t^{k+1}|^2 &= 2\int_{0}^{t}\delta X_s^{k+1}(\delta b_s \mathrm{d}t + \delta \sigma_s\mathrm{d}W_{s}) + \int_{0}^{t}|\delta \sigma_s|^2\mathrm{d}s,
			\end{align*}
			and taking the expectation
			\begin{align*}
			\mathbb{E}|\delta X_t^{k+1}|^2 &= \mathbb{E}\left[\int_{0}^{t}(2\delta X_s^{k+1}\delta b_s+|\delta \sigma_s|^2) \mathrm{d}s\right]\\
			&\leq 2\mathbb{E}\left[\int_{0}^{t}|\delta X_s^{k+1}|L(|\delta X_s^{k+1}|+|\tilde{Y}_{t}^{k}-Y_{t}^{k}| + |\tilde{Z}_{t}^{k}-Z_{t}^{k}|)\mathrm{d}s\right]\\
			&~~~+\mathbb{E}\left[\int_{0}^{t}L^2(|\delta X_s^{k+1}|+|\tilde{Y}_{t}^{k}-Y_{t}^{k}| + |\tilde{Z}_{t}^{k}-Z_{t}^{k}|)^2\mathrm{d}s\right]\\
			&\leq \mathbb{E}\left[\int_{0}^{t}((2L+L+L)|\delta X_s^{k+1}|^2+L|\tilde{Y}_{t}^{k}-Y_{t}^{k}|^2 + L|\tilde{Z}_{t}^{k}-Z_{t}^{k}|^2)\mathrm{d}s\right]\\
			&~~~+\mathbb{E}\left[\int_{0}^{t}3L^2(|\delta X_s^{k+1}|^2+|\tilde{Y}_{t}^{k}-Y_{t}^{k}|^2 + |\tilde{Z}_{t}^{k}-Z_{t}^{k}|^2)\mathrm{d}s\right]\\
			&= (4L+3L^2)\mathbb{E}\left[\int_{0}^{t}|\delta X_s^{k+1}|^2] + (L+3L^2)\mathbb{E}[\int_{0}^{t}|\tilde{Y}_{t}^{k}-Y_{t}^{k}|^2 + |\tilde{Z}_{t}^{k}-Z_{t}^{k}|^2\mathrm{d}s\right]\\
			&\leq (4L+3L^2)\mathbb{E}\left[\int_{0}^{t}|\delta X_s^{k+1}|^2] + (L+3L^2)\mathbb{E}[\int_{0}^{T}|\tilde{Y}_{t}^{k}-Y_{t}^{k}|^2 + |\tilde{Z}_{t}^{k}-Z_{t}^{k}|^2\mathrm{d}s\right],
			\end{align*}
			then we have
			\begin{align}
			\mathbb{E}|\delta X_t^{k+1}|^2 &\leq(L+3L^2)\mathbb{E}\left[\int_{0}^{T}|\tilde{Y}_{t}^{k}-Y_{t}^{k}|^2 + |\tilde{Z}_{t}^{k}-Z_{t}^{k}|^2\mathrm{d}s\right]\cdot e^{(4L+3L^2)T} \notag\\
			&=C_1 \mathbb{E}\left[\int_{0}^{T}|\tilde{Y}_{t}^{k}-Y_{t}^{k}|^2 + |\tilde{Z}_{t}^{k}-Z_{t}^{k}|^2\mathrm{d}s\right], \label{eq:Yestimate}
			\end{align}
			based on the Gronwall inequality.

			Similarly, we have
			\begin{align*}
			-\mathrm{d}|\delta Y_t^{k+1}|^2 &= -2\delta Y_t^{k+1}\cdot\mathrm{d}\delta Y_t^{k+1}-\mathrm{d}\delta Y_t^{k+1}\cdot\mathrm{d}\delta Y_t^{k+1}\\
			&= 2\delta Y_t^{k+1}(\delta f_t \mathrm{d}t - \delta Z_t^{k+1}\mathrm{d}W_{t})-|\delta Z_t^{k+1}|^2\mathrm{d}t.
			\end{align*}
			Integrating from $ t $ to $ T $,
			\begin{align*}
			|\delta Y_t^{k+1}|^2 + \displaystyle \int_{t}^{T}|\delta Z_s^{k+1}|^2\mathrm{d}s &= |\delta Y_T^{k+1}|^2+2\int_{t}^{T}\delta Y_s^{k+1}(\delta f_s \mathrm{d}t - \delta Z_s^{k+1}\mathrm{d}W_{s}),
			\end{align*}
			and taking the expectation
			\begin{align*}
			&\mathbb{E}\left[|\delta Y_t^{k+1}|^2 + \displaystyle \int_{t}^{T}|\delta Z_s^{k+1}|^2\mathrm{d}s\right]\\
			= &\mathbb{E} |\delta Y_T^{k+1}|^2+2\mathbb{E}\left[\int_{t}^{T}\delta Y_s^{k+1}\delta f_s \mathrm{d}s\right]\\
			\leq &\mathbb{E} |\delta Y_T^{k+1}|^2 + 2\mathbb{E}\left[\int_{t}^{T}|\delta Y_s^{k+1}|L(|\delta X_s^{k+1}|+|\delta Y_s^{k+1}| + |\delta Z_s^{k+1}|)\mathrm{d}s\right]\\
			\leq &\mathbb{E} |\delta Y_T^{k+1}|^2 + \mathbb{E}\left[\int_{t}^{T}L(|\delta X_s^{k+1}|^2+|\delta Y_s^{k+1}|^2)\mathrm{d}s\right] + \mathbb{E}\left[\int_{t}^{T}2L|\delta Y_s^{k+1}|^2\mathrm{d}s\right]\\
			&+ \mathbb{E}\left[\int_{t}^{T} 2L^2|\delta Y_s^{k+1}|^2 +\dfrac{1}{2}|Z_s^{k+1}|^2\mathrm{d}s\right]\\
			\leq &\mathbb{E} |\delta Y_T^{k+1}|^2 + L\cdot C_1\mathbb{E}\left[\int_{0}^{T}|\tilde{Y}_{t}^{k}-Y_{t}^{k}|^2 + |\tilde{Z}_{t}^{k}-Z_{t}^{k}|^2\mathrm{d}s\right]\\
			&+(3L+2L^2)\mathbb{E}\left[\int_{t}^{T}|\delta Y_s^{k+1}|^2\mathrm{d}s\right]+\mathbb{E}\left[\int_{t}^{T} \dfrac{1}{2}|Z_s^{k+1}|^2\mathrm{d}s\right],
			\end{align*}
			then we have
			\begin{align}\label{eq:YZestimate}
			&\mathbb{E}\left[|\delta Y_t^{k+1}|^2 + \dfrac{1}{2}\displaystyle \int_{t}^{T}|\delta Z_s^{k+1}|^2\mathrm{d}s\right]\notag
			\\
			\leq &\mathbb{E}|\delta Y_T^{k+1}|^2 + L\cdot C_1\mathbb{E}\left[\int_{0}^{T}|\tilde{Y}_{t}^{k}-Y_{t}^{k}|^2 + |\tilde{Z}_{t}^{k}-Z_{t}^{k}|^2\mathrm{d}s\right]\notag\\
			&+(3L+2L^2)\mathbb{E}\left[\int_{t}^{T}|\delta Y_s^{k+1}|^2\mathrm{d}s\right],
			\end{align}
			thus
			\begin{align*}
			\mathbb{E}|\delta Y_t^{k+1}|^2 \leq &\mathbb{E}\left[L^2|\delta X_T^{k+1}|^2\right] + L\cdot C_1\mathbb{E}\left[\int_{0}^{T}|\tilde{Y}_{t}^{k}-Y_{t}^{k}|^2 + |\tilde{Z}_{t}^{k}-Z_{t}^{k}|^2\mathrm{d}s\right]\notag\\
			&+(3L+2L^2)\mathbb{E}\left[\int_{t}^{T}|\delta Y_s^{k+1}|^2\mathrm{d}s\right]\\
			\leq&(L+L^2)C_1\mathbb{E}\left[\int_{0}^{T}|\tilde{Y}_{t}^{k}-Y_{t}^{k}|^2 + |\tilde{Z}_{t}^{k}-Z_{t}^{k}|^2\mathrm{d}s\right]\\
			&+(3L+2L^2)\mathbb{E}\left[\int_{t}^{T}|\delta Y_s^{k+1}|^2\mathrm{d}s\right].
			\end{align*}
			Based on the Gronwall inequality, we get
			\begin{align}
			\mathbb{E}|\delta Y_t^{k+1}|^2 &\leq C_1(L+L^2)\mathbb{E}\left[\int_{0}^{T}|\tilde{Y}_{t}^{k}-Y_{t}^{k}|^2 + |\tilde{Z}_{t}^{k}-Z_{t}^{k}|^2\mathrm{d}s\right]\cdot e^{(2L^2+2L)T} \notag\\
			&= C_2\mathbb{E}\left[\int_{0}^{T}|\tilde{Y}_{t}^{k}-Y_{t}^{k}|^2 + |\tilde{Z}_{t}^{k}-Z_{t}^{k}|^2\mathrm{d}s\right] \label{eq:Yestimate123},
			\end{align}
			then
			\begin{align}
			\mathbb{E}\left[\displaystyle \int_{0}^{T}|\delta Y_t^{k+1}|^2\mathrm{d}t\right] &\leq \displaystyle \int_{0}^{T}\mathbb{E}|\delta Y_t^{k+1}|^2\mathrm{d}t \notag\\
			&\leq \displaystyle \int_{0}^{T}C_3\mathbb{E}\left[\int_{0}^{T}|\tilde{Y}_{t}^{k}-Y_{t}^{k}|^2 + |\tilde{Z}_{t}^{k}-Z_{t}^{k}|^2\mathrm{d}s\right]\mathrm{d}t \notag\\
			&= TC_2\mathbb{E}\left[\int_{0}^{T}|\tilde{Y}_{t}^{k}-Y_{t}^{k}|^2 + |\tilde{Z}_{t}^{k}-Z_{t}^{k}|^2\mathrm{d}s\right]\notag \\
			&= C_3\mathbb{E}\left[\int_{0}^{T}|\tilde{Y}_{t}^{k}-Y_{t}^{k}|^2 + |\tilde{Z}_{t}^{k}-Z_{t}^{k}|^2\mathrm{d}s\right] \label{eq:Yestimate2}.
			\end{align}
			Similarly, we get
			\begin{align*}
			\mathbb{E}\left[\displaystyle \int_{0}^{T}|\delta Z_t^{k+1}|^2\mathrm{d}t\right] &\leq C_3\mathbb{E}\left[\int_{0}^{T}|\tilde{Y}_{t}^{k}-Y_{t}^{k}|^2 + |\tilde{Z}_{t}^{k}-Z_{t}^{k}|^2\mathrm{d}s\right],
			\end{align*}
			and
			\begin{align*}
			\mathbb{E}\left[\displaystyle \int_{0}^{T}(|\delta Y_t^{k+1}|^2 + |\delta Z_t^{k+1}|^2)\mathrm{d}t\right] &\leq 2C_3\mathbb{E}|\delta Y_T^{k+1}|^2.
			\end{align*}
			Note that
			\begin{align*}
			\delta Y_t^{k+1} = \delta Y_T^{k+1} + \displaystyle \int_{t}^{T}\delta f_s \mathrm{d}s - \int_{t}^{T} \delta Z_s^{k+1}\mathrm{d}W_{s},
			\end{align*}
			then
			\begin{align*}
			\sup\limits_{0\leq t\leq T}|\delta Y_t^{k+1}|^2 \leq 3|\delta Y_T^{k+1}|^2 + 3\sup\limits_{0\leq t\leq T}|\displaystyle \int_{t}^{T}\delta f_s \mathrm{d}s|^2 +3\sup\limits_{0\leq t\leq T}|\int_{t}^{T} \delta Z_s^{k+1}\mathrm{d}W_{s}|^2.
			\end{align*}
			As
			\begin{align*}
			\sup\limits_{0\leq t\leq T}|\displaystyle \int_{t}^{T}\delta f_s \mathrm{d}s|^2 &\leq (\int_{0}^{T}|\delta f_s| \mathrm{d}s)^2 \leq (\int_{0}^{T}L^2(|\delta X_s^{k+1}| + |\delta Y_s^{k+1}| + |\delta Z_s^{k+1}|) \mathrm{d}s)^2\\
			&\leq 3L^2\int_{0}^{T}(|\delta X_s^{k+1}|^2 + |\delta Y_s^{k+1}|^2 + |\delta Z_s^{k+1}|^2) \mathrm{d}s,
			\end{align*}
			we have
			\begin{align*}
			\mathbb{E}\left[\sup\limits_{0\leq t\leq T}|\delta Y_t^{k+1}|^2\right] &\leq 3\mathbb{E}|\delta Y_T^{k+1}|^2 + 9L^2\mathbb{E}\left[\int_{0}^{T}(|\delta X_s^{k+1}|^2 + |\delta Y_s^{k+1}|^2 + |\delta Z_s^{k+1}|^2) \mathrm{d}s\right]\\
			&~~~+3\mathbb{E}\left[\sup\limits_{0\leq t\leq T}|\int_{t}^{T} \delta Z_s^{k+1}\mathrm{d}W_{s}|^2\right]\\
			&\leq(3L^2C_1+9L^2(C_1+2C_3)+3C_3)\mathbb{E}\left[\int_{0}^{T}|\tilde{Y}_{t}^{k}-Y_{t}^{k}|^2 + |\tilde{Z}_{t}^{k}-Z_{t}^{k}|^2\mathrm{d}s\right]\\
			&=C_4\mathbb{E}\left[\int_{0}^{T}|\tilde{Y}_{t}^{k}-Y_{t}^{k}|^2 + |\tilde{Z}_{t}^{k}-Z_{t}^{k}|^2\mathrm{d}s\right],
			\end{align*}
			then we get
			\begin{align*}
			&\mathbb{E}\left[\sup\limits_{0\leq t\leq T}|\hat{Y}_{t}^{k+1}-Y_{t}^{k+1}|^2 +\int_{0}^{T}|\hat{Z}_{t}^{k+1}-Z_{t}^{k+1}|^2\mathrm{d}t\right] \\
			\leq &(C_4+C_3)\mathbb{E}\left[\int_{0}^{T}|\tilde{Y}_{t}^{k}-Y_{t}^{k}|^2 + |\tilde{Z}_{t}^{k}-Z_{t}^{k}|^2\mathrm{d}s\right]\\
			\leq &(C_4+C_3)(T+1)\mathbb{E}\left[\sup\limits_{0\leq t\leq T}|\tilde{Y}_{t}^{k}-Y_{t}^{k}|^2 +\int_{0}^{T}|\tilde{Z}_{t}^{k}-Z_{t}^{k}|^2\mathrm{d}t\right]\\
			= &C_5\mathbb{E}\left[\sup\limits_{0\leq t\leq T}|\tilde{Y}_{t}^{k}-Y_{t}^{k}|^2 +\int_{0}^{T}|\tilde{Z}_{t}^{k}-Z_{t}^{k}|^2\mathrm{d}t\right].
			\end{align*}
	\item[Step 2:] Given $ \tilde{Y}_{0}^{k+1} $ and $ \{\tilde{Z}_{t}^{k+1}\}_{0\leq t\leq T} $, we have
	\begin{align}
	\tilde{Y}_{t}^{k+1}& = \tilde{Y}_{0}^{k+1} - \displaystyle \int_{0}^{t}f(s,\tilde{X}_{s}^{k+1},\tilde{Y}_{s}^{k+1},\tilde{Z}_{s}^{k+1}) \mathrm{d}s + \int_{0}^{t} \tilde{Z}_{s}^{k+1}\mathrm{d}W_{s}\notag\\
	& = \tilde{Y}_{T}^{k+1} + \displaystyle \int_{t}^{T}f(s,\tilde{X}_{s}^{k+1},\tilde{Y}_{s}^{k+1},\tilde{Z}_{s}^{k+1}) \mathrm{d}s - \int_{t}^{T} \tilde{Z}_{s}^{k+1}\mathrm{d}W_{s}\label{Picard_FSDEHat}
	\end{align}
	with its solution $ (\tilde{X}_{t}^{k+1}, \tilde{Y}_{t}^{k+1}, \tilde{Z}_{t}^{k+1}) $. By using the similar method of proof with Step 1, we get
		\begin{align}\label{tildehat}
		\mathbb{E}\left[\sup\limits_{0\leq t\leq T}|\tilde{Y}_{t}^{k+1}-\hat{Y}_{t}^{k+1}|^2 +\int_{0}^{T}|\tilde{Z}_{t}^{k+1}-\hat{Z}_{t}^{k+1}|^2\mathrm{d}t\right]
		\leq C\mathbb{E}\left[|\tilde{Y}_T^{k+1} - g(\tilde{X}_T^{k+1})|^2\right],
		\end{align}
	for a constant $ C $. Then
	\begin{align*}
	&\mathbb{E}\left[\sup\limits_{0\leq t\leq T}|\tilde{Y}_{t}^{k+1}-Y_{t}^{k+1}|^2 +\int_{0}^{T}|\tilde{Z}_{t}^{k+1}-Z_{t}^{k+1}|^2\mathrm{d}t\right] \\
	\leq&\mathbb{E}\left[\sup\limits_{0\leq t\leq T}|\tilde{Y}_{t}^{k+1}-\hat{Y}_{t}^{k+1}|^2 +\int_{0}^{T}|\tilde{Z}_{t}^{k+1}-\hat{Z}_{t}^{k+1}|^2\mathrm{d}t\right]\\
	&+\mathbb{E}\left[\sup\limits_{0\leq t\leq T}|\hat{Y}_{t}^{k+1}-Y_{t}^{k+1}|^2 +\int_{0}^{T}|\hat{Z}_{t}^{k+1}-Z_{t}^{k+1}|^2\mathrm{d}t\right]\\
	\leq &C\mathbb{E}\left[|\tilde{Y}_T^{k+1} - g(\tilde{X}_T^{k+1})|^2\right]+C_5\mathbb{E}\left[\sup\limits_{0\leq t\leq T}|\tilde{Y}_{t}^{k}-Y_{t}^{k}|^2 +\int_{0}^{T}|\tilde{Z}_{t}^{k}-Z_{t}^{k}|^2\mathrm{d}t\right].
	\end{align*}
	
	Denote that
	\begin{align*}
	a_{k+1}&=\mathbb{E}\left[\sup\limits_{0\leq t\leq T}|\tilde{Y}_{t}^{k+1}-Y_{t}^{k+1}|^2 +\int_{0}^{T}|\tilde{Z}_{t}^{k+1}-Z_{t}^{k+1}|^2\mathrm{d}t\right], \\
	b_{k+1}&=C\mathbb{E}\left[|\tilde{Y}_T^{k+1} - g(\tilde{X}_T^{k+1})|^2\right],
	\end{align*}
	we have the following iterative relationships
	\begin{align*}
	a_{k+1} &\leq b_{k+1} + C_5a_{k}\\
	&\leq b_{k+1} + C_5b_{k} + C_5^{2}a_{k-1}\\
	&\leq b_{k+1} + C_5b_{k} + C_5^{2}b_{k-1} + \cdots + C_5^{k+1}b_0,
	\end{align*}
	where $ b_0 = a_0 $. Note that $ C_5<1 $ and $ \lim\limits_{k\rightarrow \infty}b_{k+1} = 0 $, we can easily proof that
	\begin{align}
	 \lim\limits_{k\rightarrow \infty}a_{k+1} = \lim\limits_{k\rightarrow \infty}\mathbb{E}\left[\sup\limits_{0\leq t\leq T}|\tilde{Y}_{t}^{k+1}-Y_{t}^{k+1}|^2+\int_{0}^{T}|\tilde{Z}_{t}^{k+1}-Z_{t}^{k+1}|^2\mathrm{d}t\right] = 0. \label{eq:limits_a}
	\end{align}
	Combining it with \eqref{Tang}, equation \eqref{SDEconditio} holds. By \eqref{eq:Yestimate}, we have
	\begin{align*}
	 \mathbb{E}|\delta X_t^{k+1}|^2 &\leq C_1 \mathbb{E}\left[\int_{0}^{T}|\tilde{Y}_{t}^{k}-Y_{t}^{k}|^2 + |\tilde{Z}_{t}^{k}-Z_{t}^{k}|^2\mathrm{d}s\right]\\
	 &\leq C_1\mathbb{E}\left[T\sup\limits_{0\leq t\leq T}|\tilde{Y}_{t}^{k}-Y_{t}^{k}|^2+\int_{0}^{T}|\tilde{Z}_{t}^{k}-Z_{t}^{k}|^2\mathrm{d}s\right]\\
	 &\leq C_1'\mathbb{E}\left[\sup\limits_{0\leq t\leq T}|\tilde{Y}_{t}^{k}-Y_{t}^{k}|^2+\int_{0}^{T}|\tilde{Z}_{t}^{k}-Z_{t}^{k}|^2\mathrm{d}s\right].
	\end{align*}
	Combining it with \eqref{eq:limits_a}, we can easily prove that \eqref{SDEconditio} holds.
	\end{itemize}
\end{proof}

\begin{rem}
	The first two methods do not need Picard iteration but requires Assumption \ref{assu:1} and \ref{assu:2}. The third method uses the Picard iteration idea, thus Assumption \ref{assu:3} needs to be satisfied in order to obtain convergence. It can handle the situations when both $ T $ and Lipschitz coefficients are small.
\end{rem}

\begin{rem}
	In case 1, we regard $Z(\cdot)$ as the feedback control of the state processes $X(\cdot),Y(\cdot)$ and only consider the initial value $Y_0$ as the control. In case 2, in order to get the optimal approximation for the whole process $Y$, we introduce a new control $u(\cdot)$ and add a new error term $|Y(\cdot)-u(\cdot)|^2$ in the cost function. When the final cost function approximate to 0, $u(\cdot)$ is exactly the state process $Y(\cdot)$. It means that we can directly get $Y(\cdot)$ at any time $t\in[0,T]$ through the trained network. Both of case 1 and case 2 do not consider the effect of process $Z(\cdot)$, while case 3 takes it into consideration. In case 3, the value of  $Z(\cdot)$ in the last iteration is used as the input of the current iteration.
\end{rem}

\section{Numerical schemes and algorithms for fully-coupled FBSDEs}\label{sec:numerical}

In Section \ref{sec:contr_persp}, FBSDE \eqref{eq:FBSDE} is transformed to different optimal control problems. Inspired by \cite{WeinanDLforBSDE} and \cite{deeplearning_FBSDE}, we use forward neural networks to simulate the control process $ Z $ and $ u $ mentioned in Section \ref{sec:contr_persp}. The universal approximation theorem \cite{approximation1989cybenko} has shown that a feed-forward network containing a finite number of neurons can approximate continuous functions for a given accuracy.

In this section,  we propose three algorithms. In the first algorithm, the inputs of the network are states $ X $ and $ Y $, and the output is $ Z $. In the second algorithm, double control processes are employed, the state $ X $ is taken as the input of the networks, the controls $ u $ and $ Z $ are taken as the output of the networks, respectively. In the third algorithm, the three items $ (X,Y,Z) $ are taken as the input of the network and the output of network is $ Z $. Here we uniformly define the feedback function as $\phi$ and assume that $\phi$ has some good properties to ensure that the discrete-time scheme of the stochastic process is convergent, according to the work of Kloeden \cite{Numerical1992Kloeden}. In our future work, we will establish concrete assumptions that $\phi$ needs to satisfy.

We firstly give the notations of the discrete time and the Brownian motion. Let $\pi$ be a partition of the time interval $ [0,T] $, where $ 0 = t_{0}<t_{1}<t_{2}<\cdots<t_{N-1}<t_{N} =T $. We define $ \Delta t_{i}=t_{i+1}-t_{i} $ and $ \Delta W_{t_{i}}=W_{t_{i+1}}-W_{t_{i}} $, where $ W_{t_{i}} \sim \mathcal{N}(0,t_{i}) $, for $ i = 0, 1, 2,\cdots, N-1 $. For different optimal control problems in Section \ref{sec:contr_persp}, we can give their corresponding discrete-time schemes.

We construct a fully connected neural network at each time point to approximate the control $ Z $. We use $ \theta_i $ to represent the parameters of the feed-forward neural network at time $ t_{i} $. We denote $ \theta = \{\theta_i\} _{0\leq i\leq N-1} $ and define $ Y_0^{k+1,\pi} $ as a parameter. The ReLU activation function and Adam stochastic gradient descent-type algorithm are adopted in the neural network. The definitions of the loss function are different according to different control problems. The network parameters are updated by back propagation (BP), and then the optimal parameters can be found. The algorithms first update the parameters from the last layer to the first layer of the network at time point $ t_{N-1} $, and then update the parameters in the reverse direction of time until time $ t_0 $.
\begin{figure}[H]
	\centering
	\includegraphics[scale=0.9]{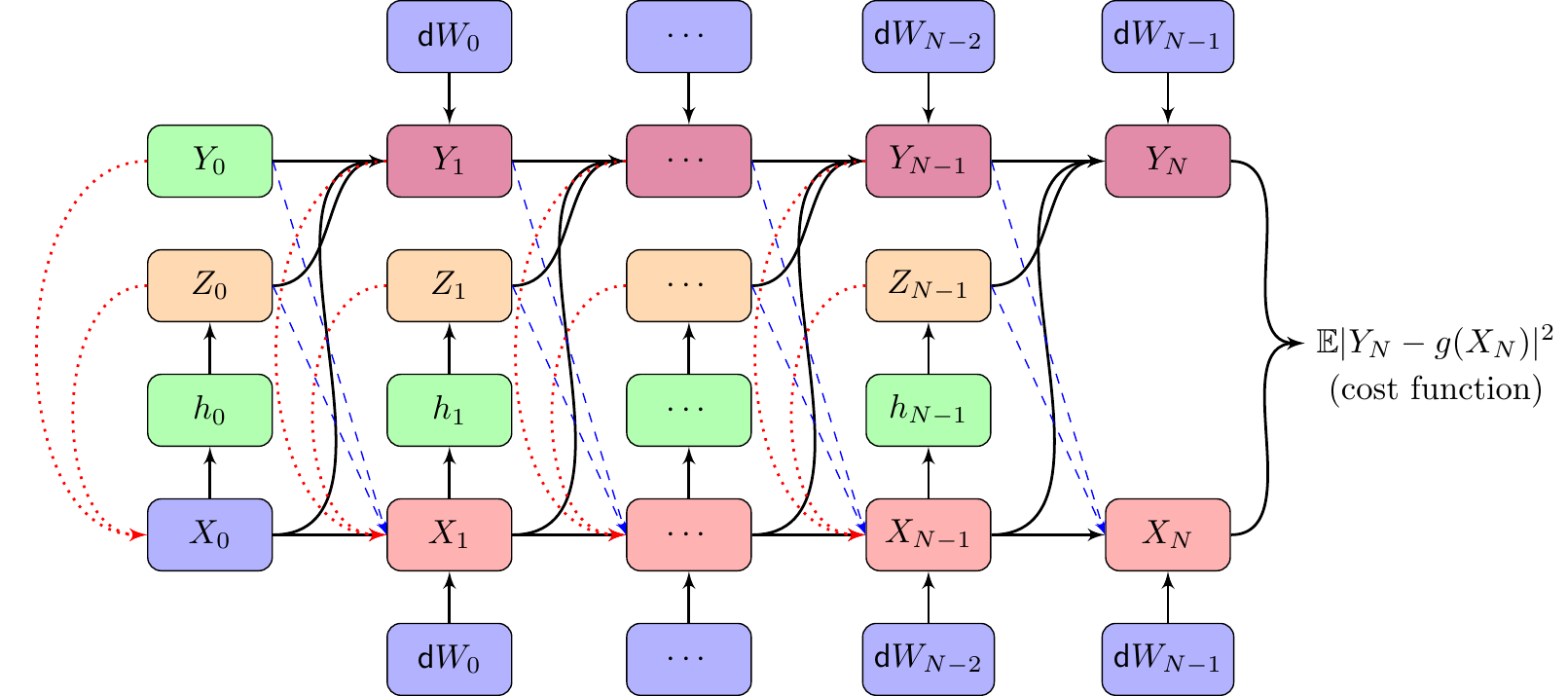}
	\caption{The neural network for Algorithm \ref{alg:1}.}
	\label{fig:whole_net}
\end{figure}

Here we take Algorithm 3 as an example, and show the neural network architecture in Figure \ref{fig:whole_net}. As shown in Figure \ref{fig:whole_net}, the solid lines represent the data flow generated in the current iteration, the dashed blue lines and dotted red lines represent respectively the data flow and neural network inputs that use the data of the previous iteration. $h_i$ represents the hidden layers and $Z_i$ represents the outputs. $Y_0$ and  the weights of the hidden layers are trainable parameters.

For convenience, the time interval $ [0,T] $ is partitioned evenly, i.e. $ \Delta t_{i} = t_{i+1} - t_{i} = T/N $ for all $ i = 0,1,2,\cdots,N-1 $. We define $ \Delta W_{i} = W_{i+1} - W_{i} $ and denote the iteration step by $ k $ which is marked by superscript in the algorithm. In the following subsections, we present these three algorithms for different kinds of feedback control in detail.

\subsection{Algorithm 1: Feedback control based on \texorpdfstring{$ (X,Y) $}{}}

In \cite{deeplearning_FBSDE}, coupled FBSDEs where the forward SDE does not depend on $Z_t$ have been studied. The values $ (X_{t_i}, Y_{t_i}, Z_{t_i})_{0\leq i\leq N-1} $ are calculated step by step in time, i.e. the triple $ (X_{t_i}, Y_{t_i}, Z_{t_i}) $ of the current time-step is used to calculate the triple $ (X_{t_{i+1}}, Y_{t_{i+1}}, Z_{t_{i+1}}) $ of the next time-step. We extend this idea to solve FBSDE \eqref{eq:FBSDE}.

For the control problem \eqref{eq:objectiveFBSDE_1}, the discrete-time scheme is
\begin{equation}\label{eq:Eular_scheme_part2}
\begin{cases}
X_{t_{i+1}}^{\pi} = X_{t_{i}}^{\pi}  + b(t_{i},X_{t_{i}}^{\pi} ,Y_{t_{i}}^{\pi} ,Z_{t_{i}}^{\pi} )\Delta t_{i} + \sigma(t_{i},X_{t_{i}}^{\pi} ,Y_{t_{i}}^{\pi} ,Z_{t_{i}}^{\pi} ) \Delta W_{t_{i}}, \\
Y_{t_{i+1}}^{\pi}  = Y_{t_{i}}^{\pi}  - f(t_{i},X_{t_{i}}^{\pi} ,Y_{t_{i}}^{\pi} ,Z_{t_{i}}^{\pi} ) \Delta t_{i}  + Z_{t_{i}}^{\pi}  \Delta W_{t_{i}},\\
X_0^{\pi} = X_0, Y_0^{\pi} = Y_0,
\end{cases}
\end{equation}
because the process $ \{X_{t_i}^{\pi},Y_{t_i}^{\pi},Z_{t_i}^{\pi}\}_{0\leq i\leq N-1}  $ is Markovian, $ Z_{t_{i}}^{\pi} $ should be represented as a function of $ (X_{t_i}^{\pi},Y_{t_i}^{\pi},Z_{t_i}^{\pi}) $
\begin{equation}\label{Z_path_simu}
Z_{t_{i}}^{\pi} = \phi_{i}(X_{t_{i}}^{\pi}, Y_{t_{i}}^{\pi}, Z_{t_{i}}^{\pi}).
\end{equation}

However, when simulating the function $ \phi_{i} $ by the neural network, , $ Z_{t_{i}}^{\pi} $ cannot be used as both input and output of the network. Notice that \eqref{Z_path_simu} is an implicit function,  assuming that its explicit form is
\begin{equation}\label{Z_path_explicit}
Z_{t_{i}}^{\pi} = \phi_{i}'(X_{t_{i}}^{\pi}, Y_{t_{i}}^{\pi}) = \phi'(X_{t_{i}}^{\pi}, Y_{t_{i}}^{\pi};\theta_{i}),
\end{equation}
which only depends on $ X_{t_{i}}^{\pi} $ and $ Y_{t_{i}}^{\pi} $. Though both of the functions $ \phi_{i} $ and $ \phi_{i}' $ are unknown , the objective of network estimation changes from approximating $ \phi_{i} $ to approximating $ \phi_{i}' $. The network contains four layers including one $(n+m) $-dim input layer, two hidden $(n+m)+10 $-dim layers and a $(m\times d) $-dim output layer. The loss function is defined as
\begin{align*}
loss = \dfrac{1}{2M}\sum\limits_{i=0}^{M}|Y_T^{k+1,\pi}-g(X_T^{k+1,\pi})|^2,
\end{align*}
where $ M $ is the number of samples.
\begin{algorithm}[H]
	\caption{Feedback control based on $ (X,Y) $}
	\label{alg:2}
	\begin{algorithmic}[1]
		\Require The Brownian motion $ \Delta W_{t_i} $, initial parameters $ (\theta^0,Y_0^{0,\pi}) $, learning rate $ \eta $;
		\Ensure The couple precess $ (X_{T}^{\pi},Y_{T}^{\pi}) $.
		\For { $ k = 0 $ to $ maxstep $}
		\State $ X_{0}^{k,\pi} = X_{0} $, $ Y_{0}^{k,\pi} = Y_{0}^{k,\pi}; $
		\For { $ i = 0 $ to $ N-1 $}
		\State $Z_{t_{i}}^{k,\pi} = \phi(X_{t_{i}}^{k,\pi}, Y_{t_{i}}^{k,\pi};\theta_{i}^{k});$
		\State $X_{t_{i+1}}^{k,\pi}=X_{t_{i}}^{k,\pi}
		+ b(t_{i},X_{t_{i}}^{k,\pi},Y_{t_{i}}^{k,\pi},Z_{t_{i}}^{k,\pi})\Delta t_{i}+ \sigma(t_{i},X_{t_{i}}^{k,\pi},Y_{t_{i}}^{k,\pi},Z_{t_{i}}^{k,\pi}) \Delta W_{t_{i}};$
		\State $Y_{t_{i+1}}^{k,\pi} = Y_{t_{i}}^{k,\pi} - f(t_{i},X_{t_{i}}^{k,\pi},Y_{t_{i}}^{k,\pi},Z_{t_{i}}^{k,\pi})\Delta t_{i} + Z_{t_{i}}^{k,\pi} \Delta W_{t_{i}};$
		\EndFor
		\State $ (\theta^{k+1},Y_0^{k+1,\pi})=(\theta^{k},Y_0^{k,\pi})-\eta\nabla(\dfrac{1}{2M}\sum\limits_{i=0}^{M}|Y_T^{k,\pi}-g(X_T^{k,\pi})|^2); $
		\EndFor
	\end{algorithmic}
\end{algorithm}
According to Proposition \ref{pro:1}, the triple $(X_{t_{i+1}}^{k,\pi}, Y_{t_{i+1}}^{k,\pi}, Z_{t_{i}}^{k,\pi})$ converges to the true solution $(X_{t}, Y_{t}, Z_{t})$ when the loss goes to zero. The corresponding algorithm is given above.

\subsection{Algorithm 2: Feedback controls based on \texorpdfstring{$ X $}{}}

As discussed in subsection \ref{subsec:3.2}, we consider $ u_t $ and $ Z_t $ as controls. The corresponding discrete-time scheme of control problem \eqref{eq:objectiveFBSDE_3} is
\begin{equation}\label{eq:Eular_scheme_part3}
\begin{cases}
X_{t_{i+1}}^{\pi} = X_{t_{i}}^{\pi}  + b(t_{i},X_{t_{i}}^{\pi} ,u_{t_{i}}^{\pi} ,Z_{t_{i}}^{\pi} )\Delta t_{i} +\sigma(t_{i},X_{t_{i}}^{\pi} ,u_{t_{i}}^{\pi} ,Z_{t_{i}}^{\pi} ) \Delta W_{t_{i}}, \\
Y_{t_{i+1}}^{\pi}  = Y_{t_{i}}^{\pi}  - f(t_{i},X_{t_{i}}^{\pi} ,Y_{t_{i}}^{\pi} ,Z_{t_{i}}^{\pi} ) \Delta t_{i}  + Z_{t_{i}}^{\pi}  \Delta W_{t_{i}},\\
X_0^{\pi} = X_0,Y_0^{\pi} = u_0,\\
u_{t_{i}}^{\pi} = \phi_{i}^{1}(X_{t_{i}}^{\pi}, u_{t_{i}}^{\pi}, Z_{t_{i}}^{\pi}),\\
Z_{t_{i}}^{\pi} = \phi_{i}^{2}(X_{t_{i}}^{\pi}, u_{t_{i}}^{\pi}, Z_{t_{i}}^{\pi}),
\end{cases}
\end{equation}
solving functions $ \phi_{i}^{1} $ and $ \phi_{i}^{2} $, we can get
\begin{equation}\label{eq:Z_path_explicit_alg3}
\begin{cases}
u_{t_{i}}^{\pi} = \phi_{i}^{'1}(X_{t_{i}}^{\pi})=\phi^{1}(X_{t_{i}}^{\pi};\theta_{i}^1),\\
Z_{t_{i}}^{\pi} = \phi_{i}^{'2}(X_{t_{i}}^{\pi})=\phi^{2}(X_{t_{i}}^{\pi};\theta_{i}^2).
\end{cases}
\end{equation}

Thus we need to construct two networks at the same time, one for simulating $ u $ and another for simulating $ Z $. The network simulating $ u $ at each time point consists of four layers including one $ n $-dim input layer, two hidden $ n+10 $-dim layers and a $ m $-dim output layer. The network layers simulating $ Z $ is the same as that of $ u $ except that the output layer is $ m\times d $-dim. All parameters of the two networks are represented as $ \theta $.

The lost function is denoted as
\begin{align*}
loss = \dfrac{1}{2M}\sum\limits_{i=0}^{M}\left[|Y_T^{\pi}-g(X_T^{\pi})|^2 + \dfrac{T}{N}\sum\limits_{j=0}^{N-1}|Y_{t_j}^{\pi}-u_{t_j}^{\pi}|^2\right].
\end{align*}
According to Proposition \ref{pro:2}, the triple $(X_{t_{i+1}}^{k,\pi}, u_{t_{i+1}}^{k,\pi}, Z_{t_{i}}^{k,\pi})$ converges to the solution $(X_{t}, Y_{t}, Z_{t})$ of \eqref{eq:FBSDE} when the loss tends to zero. The detailed algorithm is shown as follows.

\begin{algorithm}[H]
	\caption{Feedback controls based on $ X $}
	\label{alg:3}
	\begin{algorithmic}[1]
		\Require The Brownian motion $ \Delta W_{t_i} $, initial parameters $ \theta^0$, learning rate $ \eta $;
		\Ensure $ X_{T}^{k,\pi} $ and precess $ (Y_{t_i}^{k,\pi})_{0\leq i \leq N} $.
		\For { $ k = 1 $ to $ maxstep $}
		\State $ L=0; $
		\State $ X_{0}^{k,\pi} = X_{0}; $
		\State $ Y_{0}^{k,\pi} = \phi^1(X_{0};\theta_{0}^{k-1});$
		\For { $ i = 0 $ to $ N-1 $}
		\State $ u_{t_{i}}^{k,\pi} = \phi^1(X_{t_{i}}^{k,\pi}; \theta_{i}^{1,k-1});$
		\State $ Z_{t_{i}}^{k,\pi} = \phi^2(X_{t_{i}}^{k,\pi}; \theta_{i}^{2,k-1});$
		\State $X_{t_{i+1}}^{k,\pi} = X_{t_{i}}^{k,\pi} + b(t_{i},X_{t_{i}}^{k,\pi},u_{t_{i}}^{k,\pi},Z_{t_{i}}^{k,\pi})\Delta t_{i} + \sigma(t_{i},X_{t_{i}}^{k,\pi},u_{t_{i}}^{k,\pi},Z_{t_{i}}^{k,\pi}) \Delta W_{t_{i}};$
		\State $Y_{t_{i+1}}^{k,\pi} = Y_{t_{i}}^{k,\pi} - f(t_{i},X_{t_{i}}^{k,\pi},Y_{t_{i}}^{k,\pi},Z_{t_{i}}^{k,\pi})\Delta t_{i} + Z_{t_{i}}^{k,\pi} \Delta W_{t_{i}};$
		\State $ L = L + \dfrac{T}{N}|Y_{t_{i+1}}^{k,\pi}-u_{t_{i+1}}^{k,\pi}|^2; $
		\EndFor
		\State $ Loss = \dfrac{1}{2M}\sum\limits_{i=0}^{M}(|Y_T^{k,\pi}-g(X_T^{k,\pi})|^2 + L) ;$
		\State $ \theta^{k}=\theta^{k-1}-\eta\nabla Loss; $
		\EndFor
	\end{algorithmic}
\end{algorithm}

\subsection{Algorithm 3: Feedback control based on \texorpdfstring{$ (X,Y,Z) $}{}}

For control problem \eqref{eq:control1}, the corresponding discrete-time scheme can be written as
\begin{equation}\label{eq:Eular_scheme_part}
\begin{cases}
X_0^{k+1,\pi} = X_0,\\
Y_0^{k+1,\pi} = \tilde{Y}_0^{k+1},\\
X_{t_{i+1}}^{k+1,\pi} = X_{t_{i}}^{k+1,\pi} + b(t_{i},X_{t_{i}}^{k+1,\pi},Y_{t_{i}}^{k,\pi},Z_{t_{i}}^{k,\pi})\Delta t_{i} + \sigma(t_{i},X_{t_{i}}^{k+1,\pi},Y_{t_{i}}^{k,\pi},Z_{t_{i}}^{k,\pi}) \Delta W_{t_{i}}, \\
Y_{t_{i+1}}^{k+1,\pi} = Y_{t_{i}}^{k+1,\pi} - f(t_{i},X_{t_{i}}^{k+1,\pi},Y_{t_{i}}^{k+1,\pi},Z_{t_{i}}^{k+1,\pi}) \Delta t_{i} + Z_{t_{i}}^{k+1,\pi} \Delta W_{t_{i}},\\
Z_{t_{i}}^{k+1,\pi} = \phi(X_{t_{i}}^{k+1,\pi}, Y_{t_{i}}^{k,\pi}, Z_{t_{i}}^{k,\pi}; \theta_i),
\end{cases}
\end{equation}
where $ X_{t_{i}}^{k,\pi},Y_{t_{i}}^{k,\pi} $ and $ Z_{t_i}^{k+1,\pi} $ are time discretization schemes of $ \tilde{X}_{t}^{k},\tilde{Y}_{t}^{k}, \tilde{Z}_{t}^{k+1} $, respectively, and the values of  $ \{Y_{j}^{0,\pi}\} $ and $ \{Z_{j}^{0,\pi}\} $ for $ j = t_{0},t_{1},\cdots,t_{N-1} $ are given. The network at each time point $ t_i $ consists of four layers including one $ (n+m+m\times d) $-dim input layer, two hidden $ (n+m+m\times d+10) $-dim layers and a $ m\times d $-dim output layer. The loss function is defined as
\begin{align*}
loss = \dfrac{1}{2M}\sum\limits_{i=0}^{M}|Y_T^{k+1,\pi}-g(X_T^{k+1,\pi})|^2,
\end{align*}
where $ M $ is the number of samples.

According to Theorem \ref{thm:main}, the triple $(X_{t_{i+1}}^{k+1,\pi}, Y_{t_{i+1}}^{k+1,\pi}, Z_{t_{i+1}}^{k+1,\pi})$ converges to the true solution $(X_{t}, Y_{t}, Z_{t})$ when the loss tends to zero.

The detailed algorithm based on the conclusions of section \ref{sec:contr_persp} is given as following.
\begin{algorithm}[H]
	\caption{Feedback control based on $ (X,Y,Z) $}
	\label{alg:1}
	\begin{algorithmic}[1]
		\Require The Brownian motion $ \Delta W_{t_i} $, initial parameters $ (\theta^0,Y_0^{0,\pi}) $, learning rate $ \eta $ and the couple precess $ (Y_{t_i}^{0,\pi},Z_{t_i}^{0,\pi}) $;
		\Ensure The couple precess $ (Y_{t_i}^{k+1,\pi},Z_{t_i}^{k+1,\pi}) $.
		\For { $ k = 0 $ to $ maxstep $}
		\State $ X_{0}^{k+1,\pi} = X_{0}, Y_{0}^{k+1,\pi} = Y_{0}^{k,\pi}; $
		\For { $ i = 0 $ to $ N-1 $}
		\State $X_{t_{i+1}}^{k+1,\pi} = X_{t_{i}}^{k+1,\pi} +b(t_{i},X_{t_{i}}^{k+1,\pi},Y_{t_{i}}^{k,\pi},Z_{t_{i}}^{k,\pi})\Delta t_{i}+ \sigma(t_{i},X_{t_{i}}^{k+1,\pi},Y_{t_{i}}^{k,\pi},Z_{t_{i}}^{k,\pi}) \Delta W_{t_{i}};$
		\State $Z_{t_{i}}^{k+1,\pi} = \phi(X_{t_{i}}^{k+1,\pi}, Y_{t_{i}}^{k,\pi}, Z_{t_{i}}^{k,\pi}; \theta_{i}^{k});$
		\State $Y_{t_{i+1}}^{k+1,\pi} = Y_{t_{i}}^{k+1,\pi} - f(t_{i},X_{t_{i}}^{k+1,\pi},Y_{t_{i}}^{k+1,\pi},$ $Z_{t_{i}}^{k+1,\pi})\Delta t_{i} + Z_{t_{i}}^{k+1,\pi} \Delta W_{t_{i}};$
		\EndFor
		\State $ (\theta^{k+1},Y_0^{k+1,\pi})=(\theta^{k},Y_0^{k,\pi})-\eta\nabla(\dfrac{1}{2M}\sum\limits_{i=0}^{M}|Y_T^{k+1,\pi}-g(X_T^{k+1,\pi})|^2); $
		\EndFor
	\end{algorithmic}
\end{algorithm}

\begin{rem}
	In this algorithm, the Brownian motion is denoted as $ W $. The initial paths $ Y_{t_i}^{0,\pi} $ and $ Z_{t_i}^{0,\pi} $ are generated randomly, and they do not influence the convergence of $ Y_0^{k,\pi} $ and the process $ Z^{k,\pi} $. As the aim of the algorithm is to find the optimal parameters $ \theta^{*} $ to approximate the map $ \phi $ in equations \eqref{eq:Eular_scheme_part}, the gradient descent method is used in line 8 of the algorithm which makes $ \theta^{k} $ approximate to $ \theta^{*} $ after $k$ times of iterations. The numerical results shown in Section \ref{sec:results} confirm the convergence of the algorithm. Besides, when the function $ \phi $ in \eqref{eq:Eular_scheme_part} only depends on $ X_{t_i} $, the numerical results also demonstrate good convergence for the neural network approximation.
\end{rem}

\begin{rem}
    For any given FBSDE, even the assumptions in Section 3 are not satisfied, we can still use our method to calculate the numerical solution of FBSDE \eqref{eq:FBSDE}. The simulation solution is close enough to the real solution of FBSDE \eqref{eq:FBSDE} when the value of loss function is close enough to zero.
\end{rem}

\section{Numerical results}\label{sec:results}

In this section, we present the numerical results of our algorithms for different cases including partially-coupled cases and fully-coupled cases. If not mentioned, the results in the examples are the results of algorithm 3. All the examples of this section are implemented with 256 sample-paths of the Brownian motion $ W $, learning rate $ \eta = 5\times10^{-3} $, the number of time-points $ N=25 $ if not specifically noted. The numerical experiments are performed in PYTHON on a LENOVO computer with a 2.40 Gigahertz (GHz) Inter Core i7 processor and 8 gigabytes (GB) random-access memory (RAM).

\subsection{Example 1. Partially-coupled case (BSDE)}

We consider an example in \cite{explictresultsimulation} for solving of partially-coupled FBSDEs.

Assume $ t\in[0,T],x=(x_{1},\cdots,x_{d})\in \mathbb{R}^{d},y\in\mathbb{R},z\in \mathbb{R}^{d}$ and the functions $ b,\sigma,f,g $ in \eqref{eq:FBSDE} satisfy
\begin{align*}
b(t,x,y,z) &= 0, \\
\sigma(t,x,y,z) &= 0.25diag(x),\\
f(t,x,y,z) &= 0.25 \times (y - \dfrac{2 + 0.25^2 \times d}{2\times0.25^2\times d})(\sum\limits_{i=1}^{d}z_{i}),\\
g(x) &= \dfrac{exp(T+\sum_{i=1}^{d}x_{i})}{1 + exp(T+\sum_{i=1}^{d}x_{i})},
\end{align*}
where $ diag(x) $ represents a diagonal matrix where the value of the $ i $th diagonal element is $ x_{i} $ , and the explicit solution of this FBSDE is
$$ Y(t,x) = \dfrac{exp(t+\sum_{i=1}^{d}x_{i})}{1+exp(t+\sum_{i=1}^{d}x_{i})}.$$

We set $ T=0.5, X_{0} = 0 $ for different dimensions, and the explicit solution of $ Y_{0} $ is 0.5.

Figure \ref{ex_01_d1} shows that in the case of $ d=1 $, the network solution is close to the explicit solution when the number of iteration steps increases. After 10000 steps, the value of $ Y_{0} $ is 0.50496 and has a relative error of $ 0.99\% $ comparing to the explicit solution.

\begin{figure}[H]
	\centering
	\includegraphics[scale=0.45]{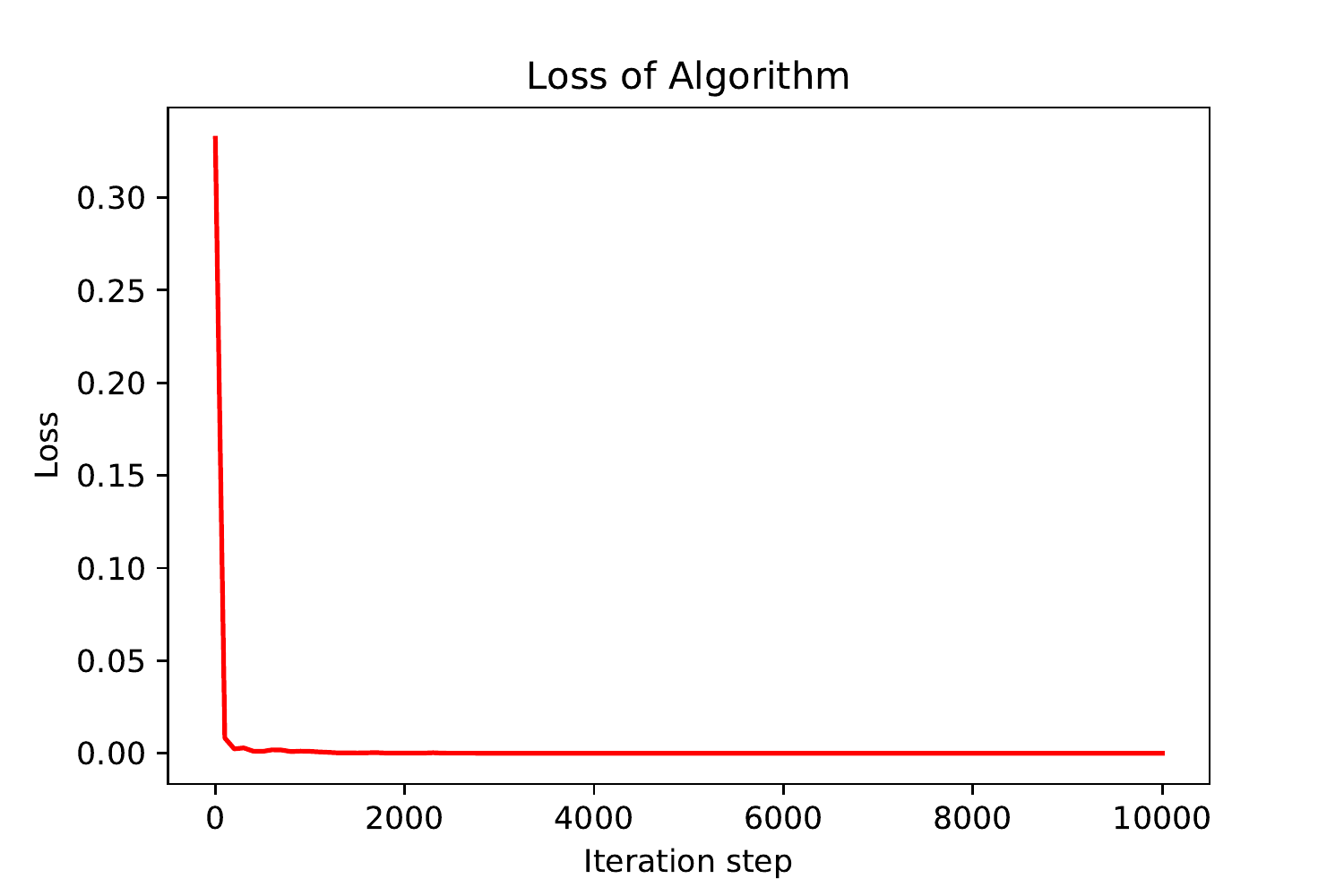}
	\includegraphics[scale=0.45]{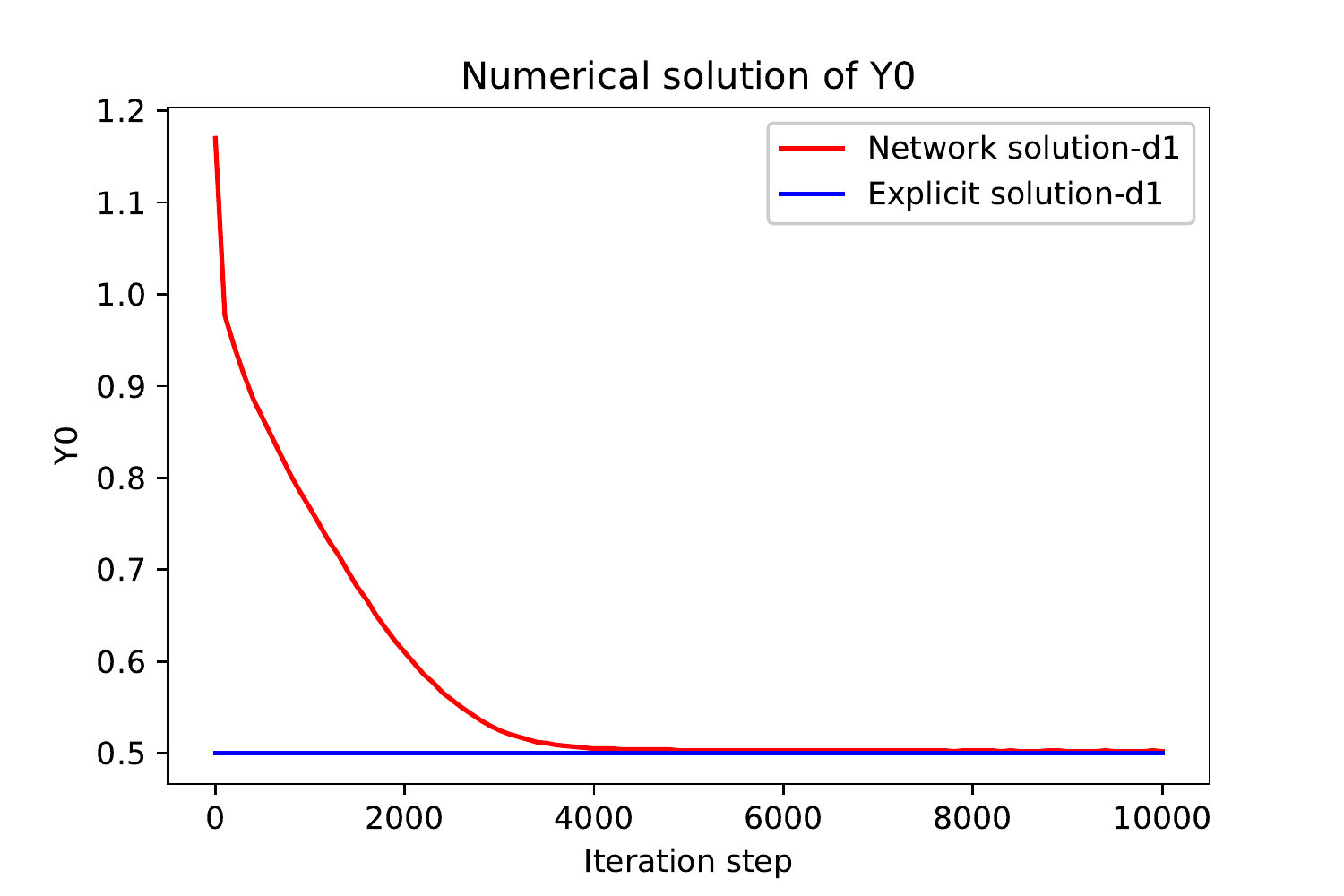}
	\caption{Case d=1. The upper figure represents the curve of loss when the iteration step increases, and the loss achieves $ 4.52\times10^{-7} $ after 10000 iteration steps. The lower figure shows the comparison between the network solution and the explicit solution which are represented with the red curve and the blue curve, respectively. }
	\label{ex_01_d1}
\end{figure}

For $ d=100 $ case, the value of $ Y_{0} $ is 0.50413 after 15000 steps. The network has a rather optimistic performance with a relative error of $ 0.83\% $ to the explicit solution, see figure \ref{ex_01_d100} in detail.

\begin{figure}[H]
	\centering
	\includegraphics[scale=0.45]{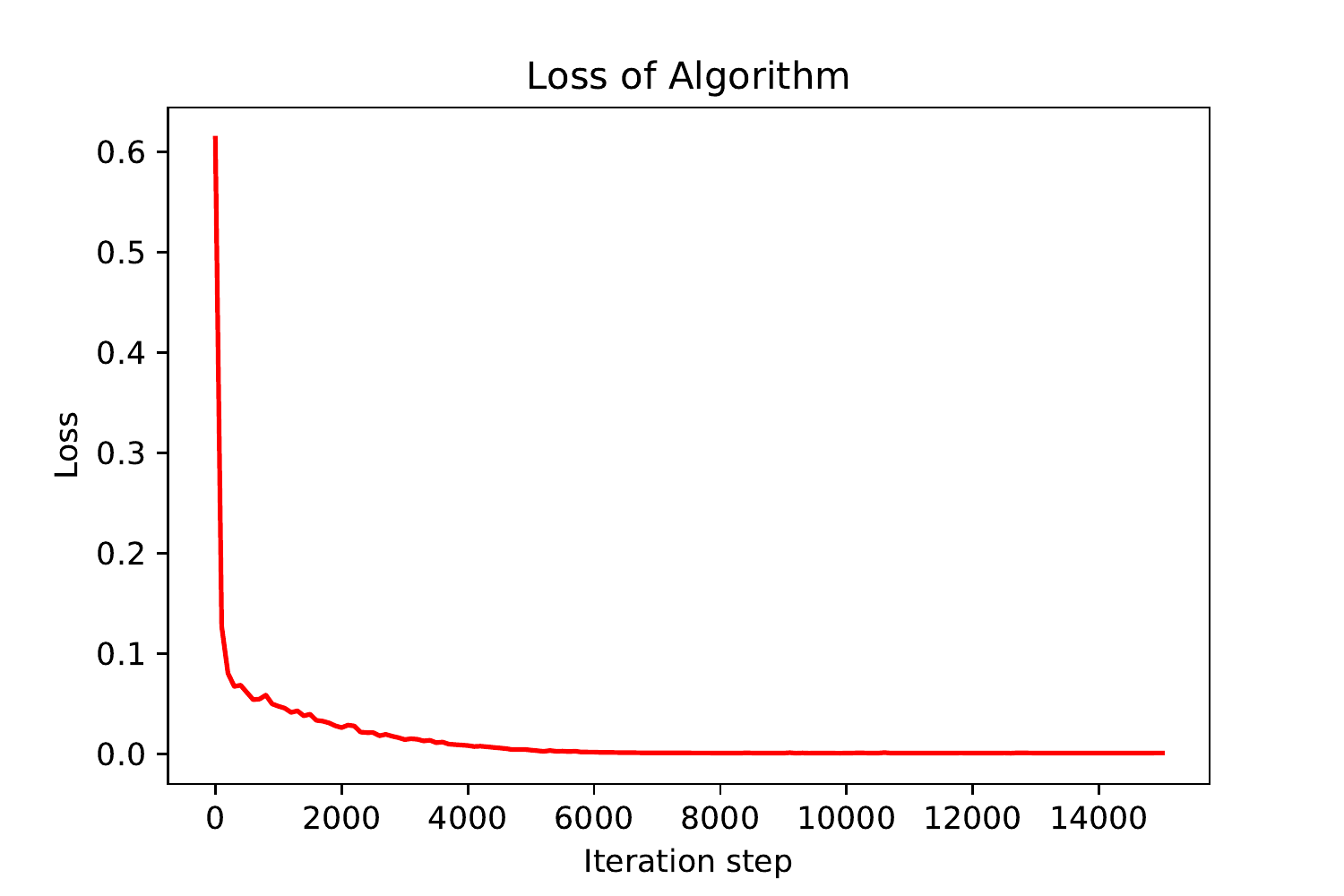}
	\includegraphics[scale=0.45]{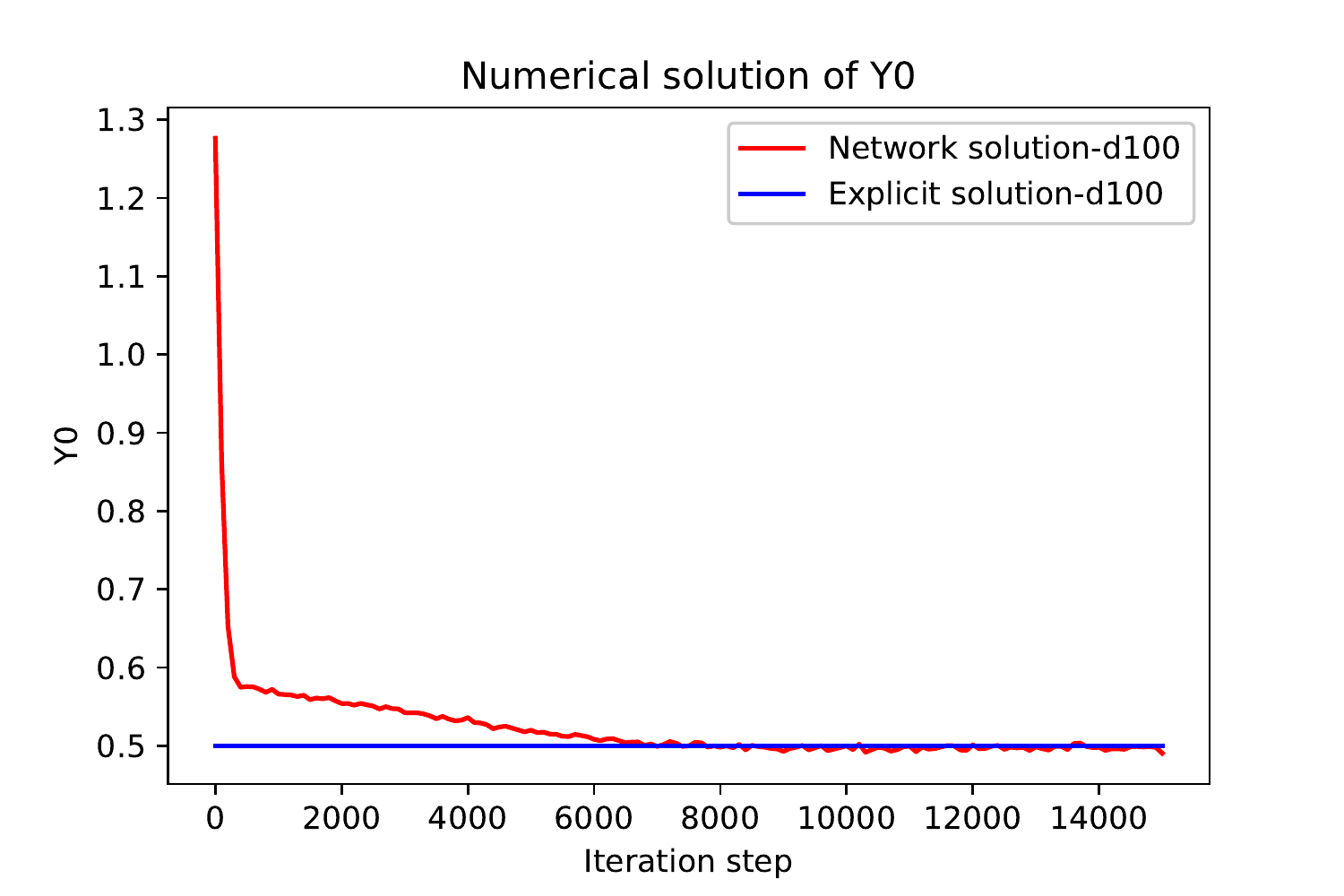}
	\caption{Case d=100. The curve of loss in the upper figure shows that after 15000 steps, the loss value is $ 3.30\times10^{-3} $. As mentioned above, the red and blue curve represent respectively the network solution and the explicit solution.}
	\label{ex_01_d100}
\end{figure}

The initial value of $ Y_0 $ is selected randomly in the interval $ [1.0, 2.0] $ and the numerical algorithm is performed 10 times independently. Table \ref{tab:Decoupled100} shows the detailed numerical results.

\begin{table}[H]
	\centering
	\caption{Numerical results comparing with the explicit solutions}
	\label{tab:Decoupled100}
	\begin{tabular}{|c|c|c|c|c|}
		\hline
		Step&Mean of $ Y_0 $&Variance of $ Y_0 $&Relative error of $ Y_0 $&Mean of runtime(s) \\
		\hline
		3000	&0.54501&3.540E-02&0.0900&4715.6\\
		\hline
		6000	&0.50970&6.803E-04&0.0194&8751.3\\
		\hline
		9000	&0.49882&5.049E-06&0.0024&12693.2\\
		\hline
		12000	&0.49658&3.752E-07&0.0068&16615.4\\
		\hline
		15000	&0.49602&4.369E-07&0.0080&20528.8\\
		\hline
	\end{tabular}
\end{table}

\subsection{Example 2. The forward SDE not containing \texorpdfstring{$ Z $}{} term}

We adopt the example of \cite{Yu2016Efficient} which does not contain $ Z $ term in the forward SDE. Consider the following FBSDE,
\begin{equation*}
\left\{
\begin{array}{l}
X_{t} = x +  \int_{0}^{t} b(s,X_{s},Y_{s}) ds + \int_{0}^{t} \sigma(s,X_{s},Y_{s}) dW_{s},\\
Y_{t} = g(X_{T}) +  \int_{t}^{T} f(s,X_{s},Y_{s},Z_{s}) ds - \int_{t}^{T} \left\langle Z_{s}, dW_{s} \right\rangle _{\mathbb{R}^{d}},
\end{array}
\right.
\end{equation*}
where
\begin{align*}
b_{i}(t,x,y) &= \dfrac{t}{2} \cos^{2}(y+x_{i}),\\
\sigma_{i,i}(t,x,y) &= \dfrac{t}{2} \sin^{2}(y+x_{i}),\\
g(x) &= \dfrac{1}{d}(\sum\limits_{i=1}^{d-1}x_{i}^{2}(x_{i+1}+T) + x_{d}^{2}(x_{1}+T)),
\end{align*}
and
\begin{align*}
&f(t,x,y,z) =\\
& \sum_{i=1}^{d}z_{i} - \dfrac{1}{d}(1+\dfrac{t}{2})\sum_{i=1}^{d}x_{i}^{2} - \dfrac{t}{d} (\sum_{i=1}^{d-1}x_{i}(x_{i+1}+t)+x_{q}(x_{1}+t))\\
&-\dfrac{t^{2}}{2d^{2}}(\sum_{i=1}^{d-1}(x_{i+1}+t)\sin^{4}(y+x_{i}) + (x_{1}+t)\sin^{4}(y+x_{d})).
\end{align*}
The explicit solution of this FBSDE is
$$ Y_{t} = \dfrac{1}{d}(\sum\limits_{i=1}^{d-1}X_{t,i}^{2}(X_{t,i+1}+t) + X_{t,d}^{2}(X_{t,1}+t)). $$

\cite{Yu2016Efficient} has shown the results of different dimensions for $ d=2,3,4,5 $. In the case  $ d=5, x=1.0 $, \cite{Yu2016Efficient} has achieved a relative error of $ 5.489\times10^{-4} $ in 134 seconds, and we get an approximated result of 1.0002 for $ Y_{0} $ with a relative error of $ 0.02\% $ comparing with the explicit solution of 1.0. Figure \ref{ex_02_d6} shows the details.

\begin{figure}[H]
	\centering
	\includegraphics[scale=0.45]{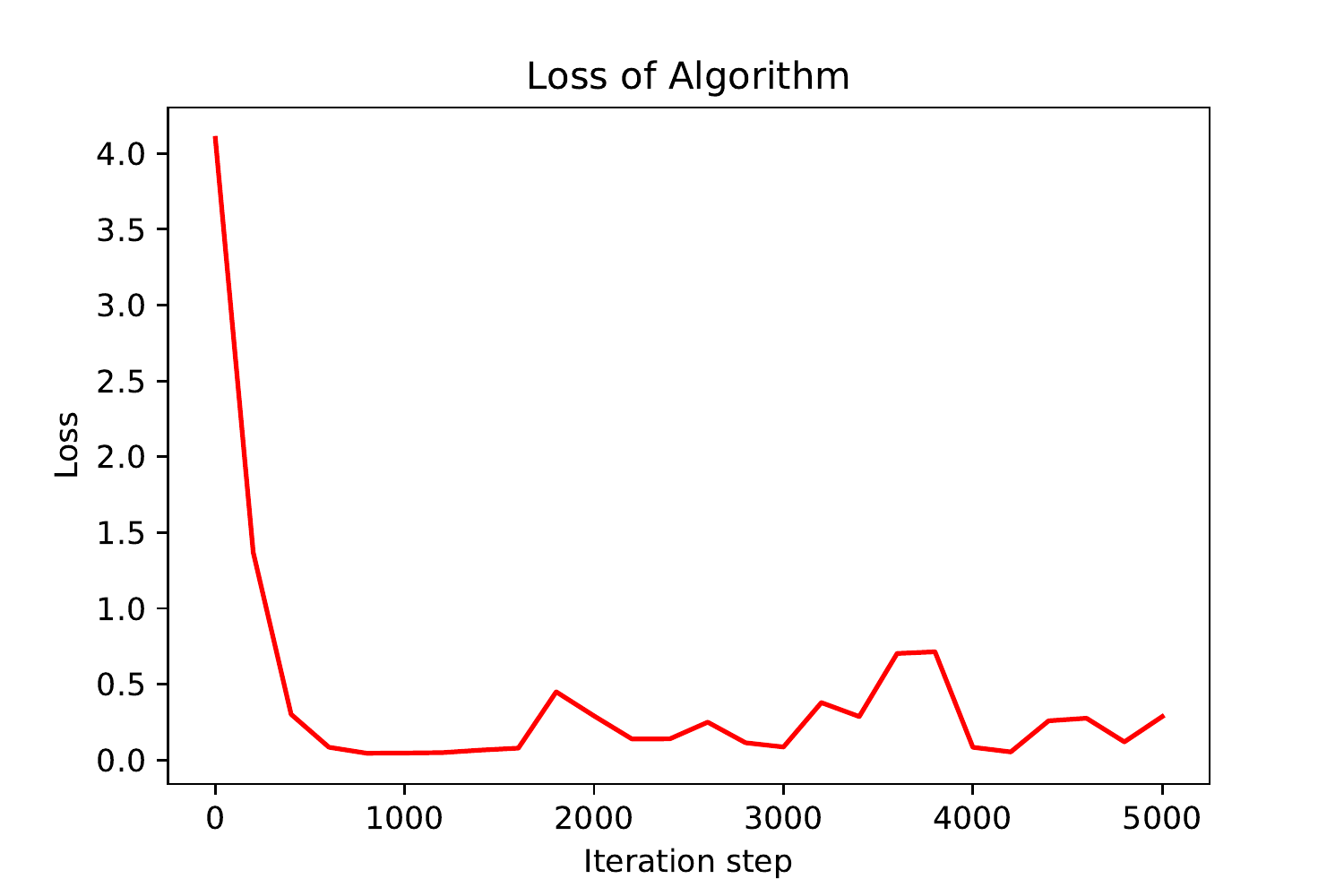}
	\includegraphics[scale=0.45]{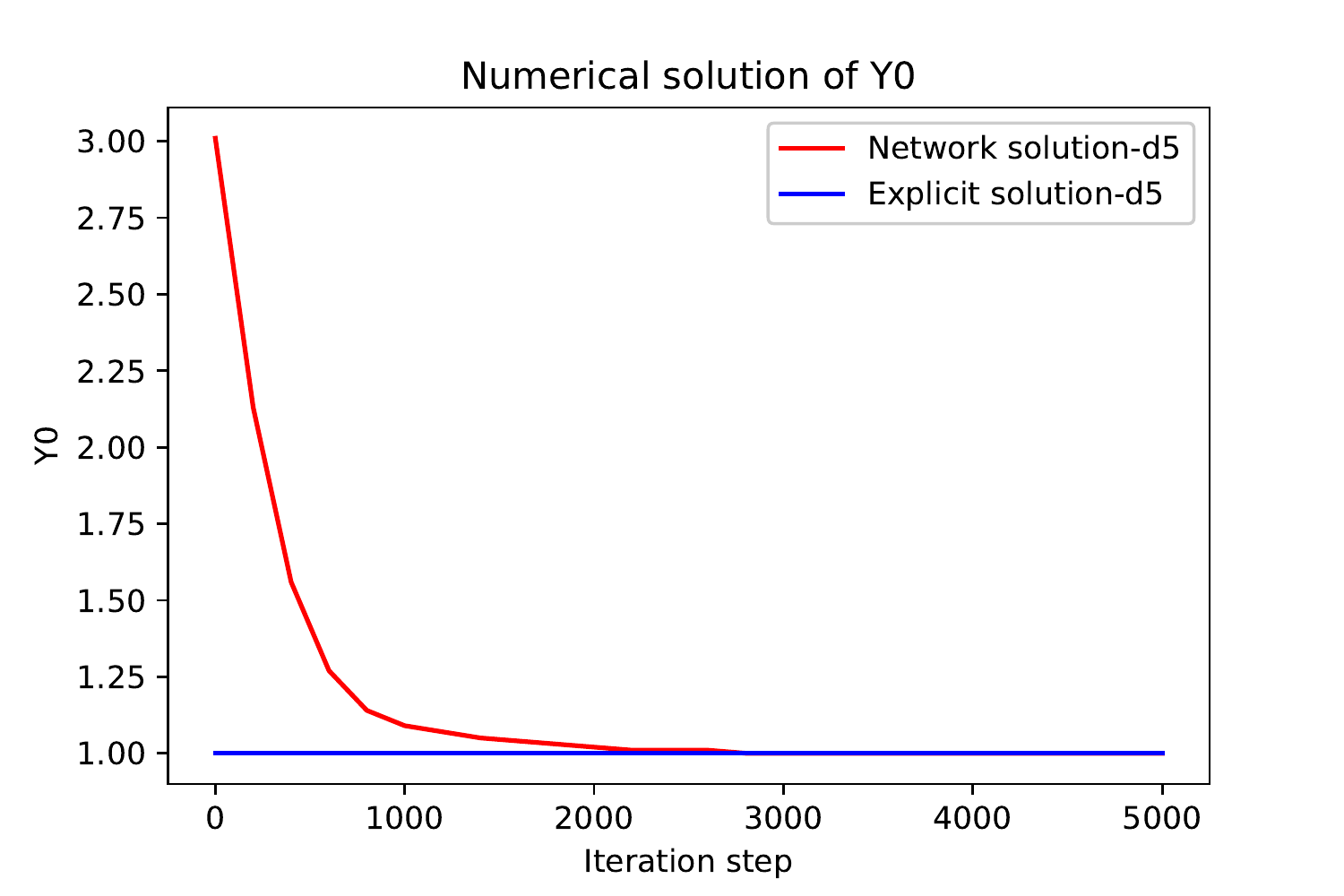}
	\caption{Case d=5. The curve of loss for 5000 iteration steps is shown in the upper and the comparison with the explicit solution is shown in the right figure. As mentioned above, the red and blue curve represent the results of the network solution and the explicit solution respectively.}
	\label{ex_02_d6}
\end{figure}

Similarly, we perform 10 independent runs for the case $ d=5 $. Detailed results are shown in Table \ref{tab:Patrial}.

\begin{table}[H]
	\centering
	\caption{Numerical results comparing with the explicit solutions}
	\label{tab:Patrial}
	\begin{tabular}{|c|c|c|c|c|}
		\hline
		Step&Mean of $ Y_0 $&Variance of $ Y_0 $&Relative error of $ Y_0 $&Mean of runtime(s) \\
		\hline
		1000	&1.09393&6.445E-02&9.393E-02&279.9\\
		\hline
		2000	&1.02127&1.046E-04&2.127E-02&461.4\\
		\hline
		3000	&1.00247&7.194E-06&2.470E-03&658.2\\
		\hline
		4000	&1.00025&8.857E-06&2.538E-04&880.9\\
		\hline
		5000	&1.00020&8.000E-06&1.996E-04&1077.4\\
		\hline
	\end{tabular}
\end{table}

For high dimensional cases, the method of \cite{Yu2016Efficient} is not applicable while our neural network method shows satisfactory results. For the case $ d=100 $, our network demonstrates remarkable performance and the relative error of $ Y_{0} $ is $ 0.1\% $. See Figure \ref{ex_02_d100} in detail.

\begin{figure}[H]
	\centering
	\includegraphics[scale=0.45]{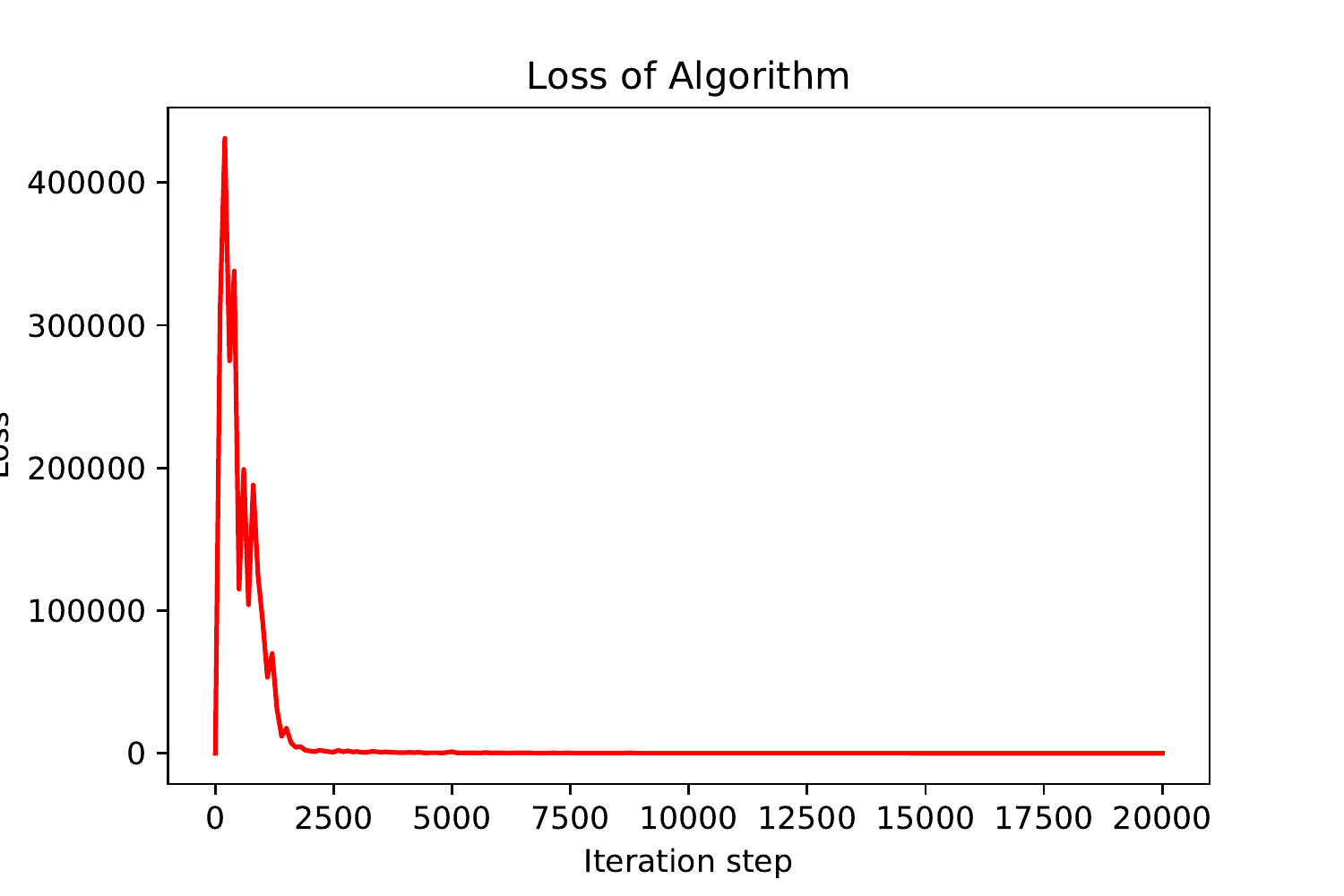}
	\includegraphics[scale=0.45]{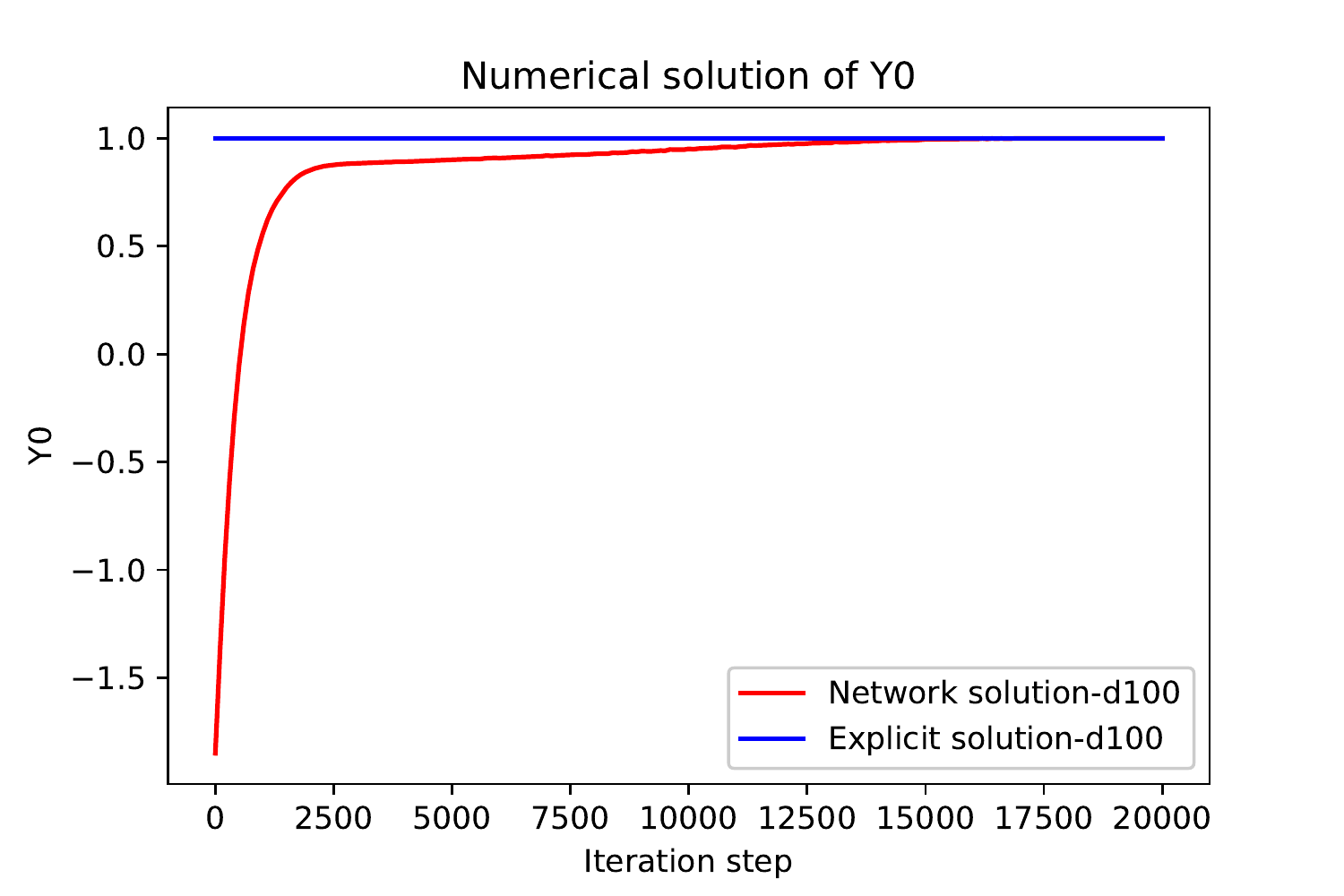}
	\caption{Case d=100. the upper figure shows the curve of loss when the number of iteration steps increases. The lower figure shows the comparison between the network solution and the explicit solution, the relative error has a downward trend when the number of iteration steps increases and tends to be stable at about $ 0.1\% $. }
	\label{ex_02_d100}
\end{figure}

\subsection{Example 3. 1-dim fully-coupled case}

We adopt the example of ~\cite{Zhao2016Multistep} for the fully-coupled case. Considering the following FBSDE,
\begin{empheq}[left=\empheqlbrace]{align*}
X_{t} =& X_{0} -  \begin{matrix}
\int_{0}^{t}\dfrac{1}{2} \sin(s+X_{s}) \cos(s+X_{s})(Y_{s}^{2}+Z_{s})ds
\end{matrix} \\
&+ \begin{matrix}
\int_{0}^{t}\dfrac{1}{2}\cos(s+X_{s})(Y_{s}\sin(s+X_{s})+Z_{s} +1)dW_{s}
\end{matrix},\\
Y_{t} =& \sin(T+X_{T}) +  \begin{matrix}
\int_{t}^{T}Y_{s}Z_{s}-\cos(s+X_{s})ds
\end{matrix} - \begin{matrix}
\int_{t}^{T}Z_{s}dW_{s},
\end{matrix}
\end{empheq}
the solutions of this FBSDE are $ Y(t,X_{t}) = \sin(t+X_{t}) $ and $ Z(t,X_{t}) = \cos^{2}(t+X_{t}) $. The numerical solution for this kind of FBSDEs is much more difficult as its diffusion coefficient is dependent on $ Z. $

We set $ X_{0} = 1.0, T=0.1 $, and the explicit solution $ Y_{0} \approx 0.84147 $. The numerical results are shown in Figure \ref{ex_03_d5} which are close to the results of \cite{Zhao2016Multistep}, and the relative error is $ 7.49\times 10^{-4} $.

\begin{figure}[H]
	\centering
	\includegraphics[scale=0.45]{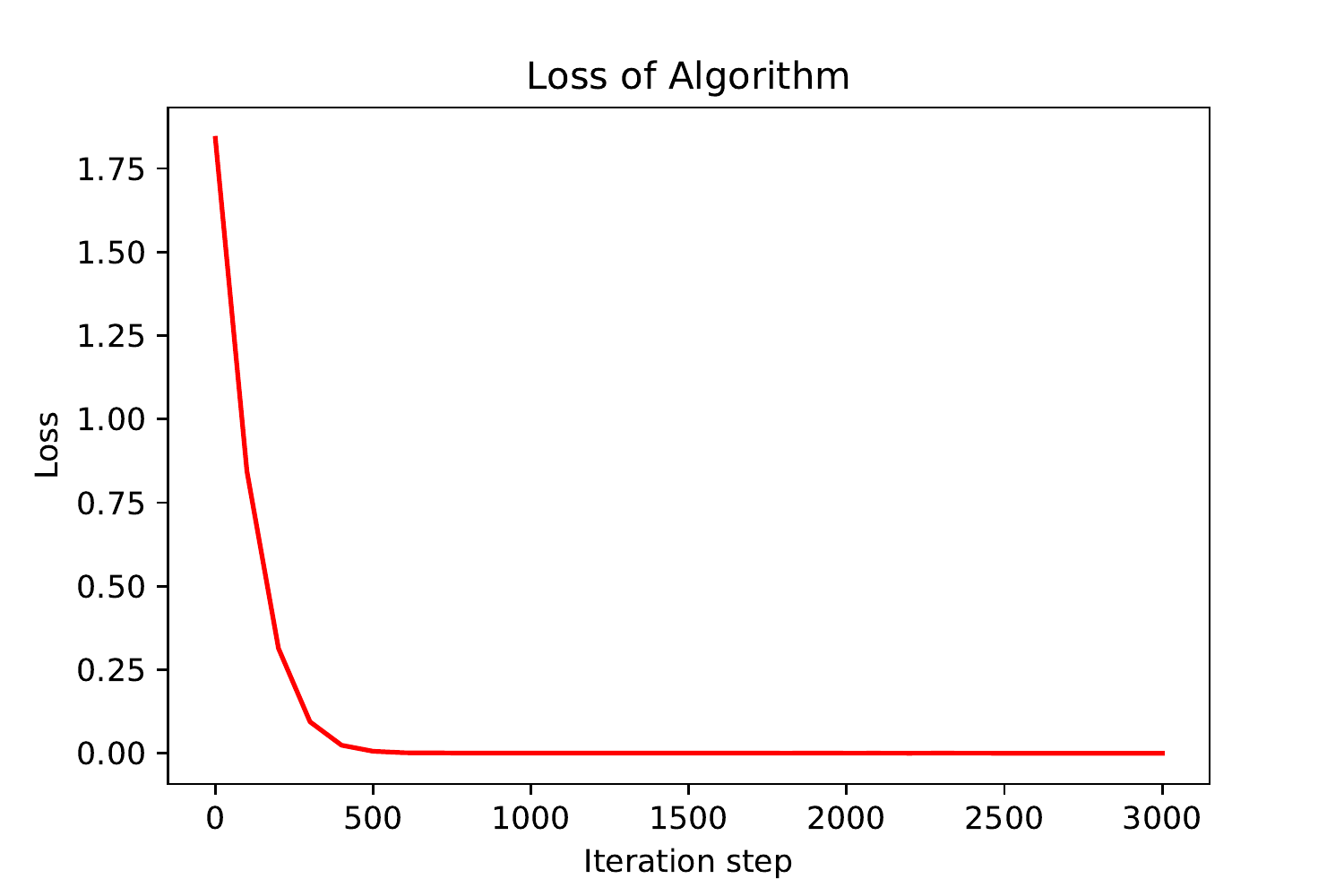}
	\includegraphics[scale=0.45]{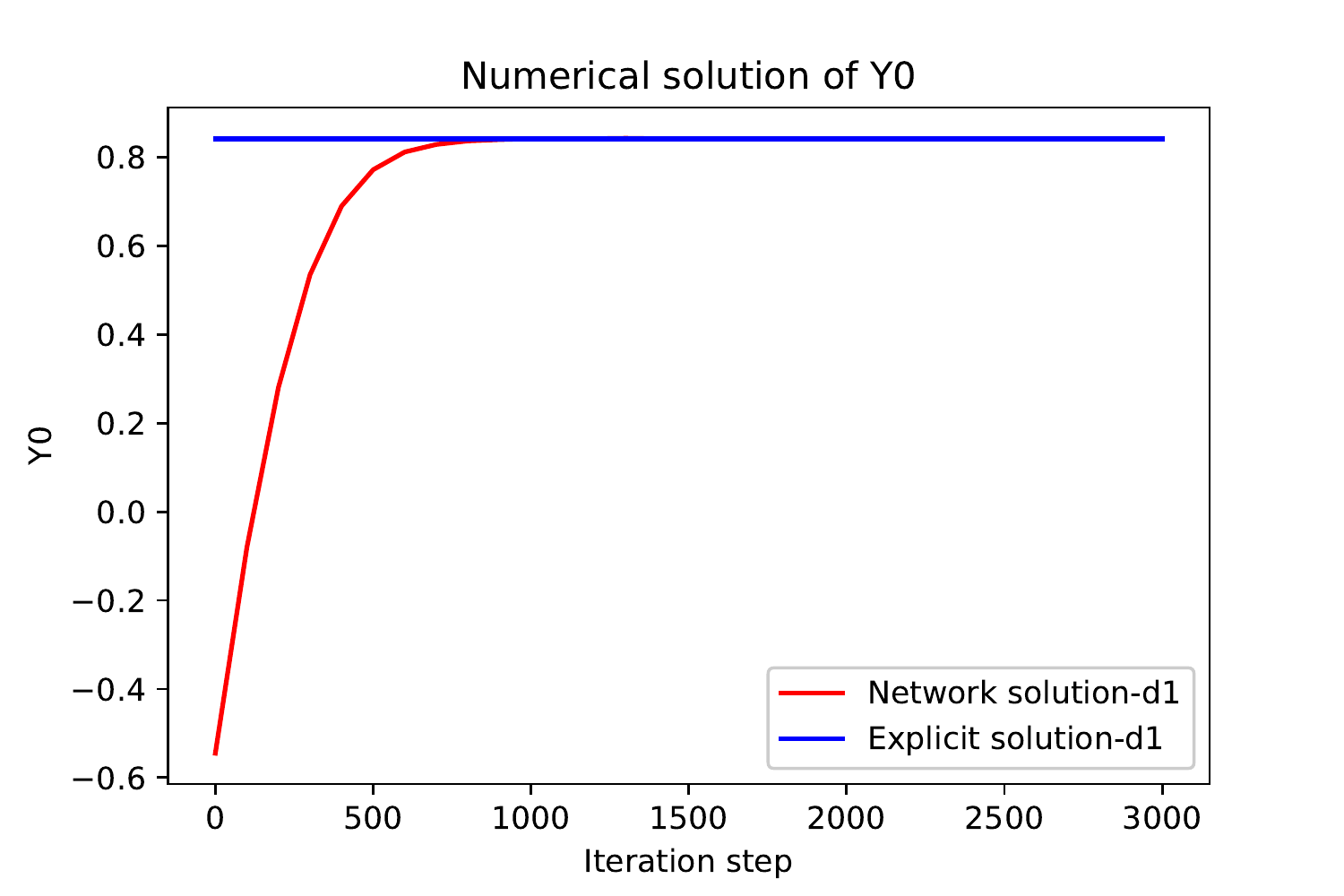}
	\caption{Case d=1. After 3000 iteration steps, the loss is reduced to $ 2.49\times 10^{-4} $, and the relative error is reduced to $ 7.49\times 10^{-4}$. }
	\label{ex_03_d5}
\end{figure}

In this example, we also use Algorithm 1 and 2 to solve the FBSDE. We use the same time partition, learning rate, number of iteration steps and samples as in Algorithm 3. In order to compare the results of the three algorithms more throughly, we calculate the variance of 1000 iteration steps before the current iteration. When the variance is less than $ 1\times 10^{-7} $ or the number of iteration steps achieves the upper limit 10000, the loop iteration is terminated. We record the number of iteration steps and the running time when the algorithm terminate. As expected, all the results are convergent. The comparison of the three algorithms are shown in Table 3. We run each algorithm 10 times independently.

\begin{table}[H]
	\centering
	\caption{ Numerical results comparison of the 3 algorithms }
	\label{tab:compared1}
	\begin{tabular}{|c|c|c|c|c|c|}
		\hline
		Method&Mean of $ Y_0 $&Variance of $ Y_0 $&Relative error of $ Y_0 $&Steps&Time(s) \\
		\hline
		Alg 1	&0.8381&1.262E-05&4.32E-03&2588.7&390.9\\
		\hline
		Alg 2	&0.8334&5.367E-06&9.86E-03&10000&2160.0\\
		\hline
		Alg 3	&0.8421&5.742E-08&4.24E-04&2030.4&1076.4\\
		\hline
	\end{tabular}
\end{table}

From Table 3, we can see that Algorithm 3 is more accurate and stable for this example. It requires the least number of iteration steps to achieve a smooth convergence result but it takes the longest time for each iteration step. Algorithm 1 and 2 takes less running time for each iteration step, but Algorithm 2 can not get a stable convergence result up to the maximum iteration step.

\subsection{Example 4. 100-dim nonlinear generator for the FBSDE}

Assume $ t\in[0,T],x=(x_{1},\cdots,x_{d})\in \mathbb{R}^{d},y\in\mathbb{R},z\in \mathbb{R}^{d}$ , consider this following FBSDE,
\begin{equation}\label{eq:exmp:1}
\begin{cases}
\mathrm{d}X_{i,t}= d\exp(-\frac{1}{d}\sum_{i=0}^{d} X_{i,t}) Z_{i,t} \mathrm{d}W_{i,t},\\
-\mathrm{d}Y_{t} = -\exp(-\frac{1}{d}\sum_{i=0}^{d} X_{i,t})(\sum_{i=1}^{d}Z_{i,t}^2)\mathrm{d}t - Z_{t}\transpose \mathrm{d}W_{t},\\
X_0 = x, Y_T = \exp(\frac{1}{d}\sum_{i=0}^{d} X_{i,T}),
\end{cases}
\end{equation}
where $ \mathrm{d}W $ takes value in $ \mathbb{R}^{d} $. We can check that the explicit solution of this FBSDE is
$$ Y(t,x) = \exp(\dfrac{1}{d}\sum_{i=0}^{d} X_{i,t}).$$

We set $ T=0.1, X_{0} = 1.0 $ for $ d=100 $, the explicit solution of $ Y_{0} $ is $ e \approx 2.7183 $, and the number of time  points is $ N=25 $. The loss curve and the numerical solution of $ Y_{0} $ are shown in Figure \ref{fig:ex_01_d100}.We can see from Figure \ref{fig:ex_01_d100} that in the case of $ d=100 $, the result of the network solution is closer to the explicit solution when the number of iteration steps increases. After 4000 steps, the value of $ Y_{0} $ is 2.71662 and has a relative error of $ 0.061\% $.

\begin{figure}[H]
	\centering
	\includegraphics[scale=0.45]{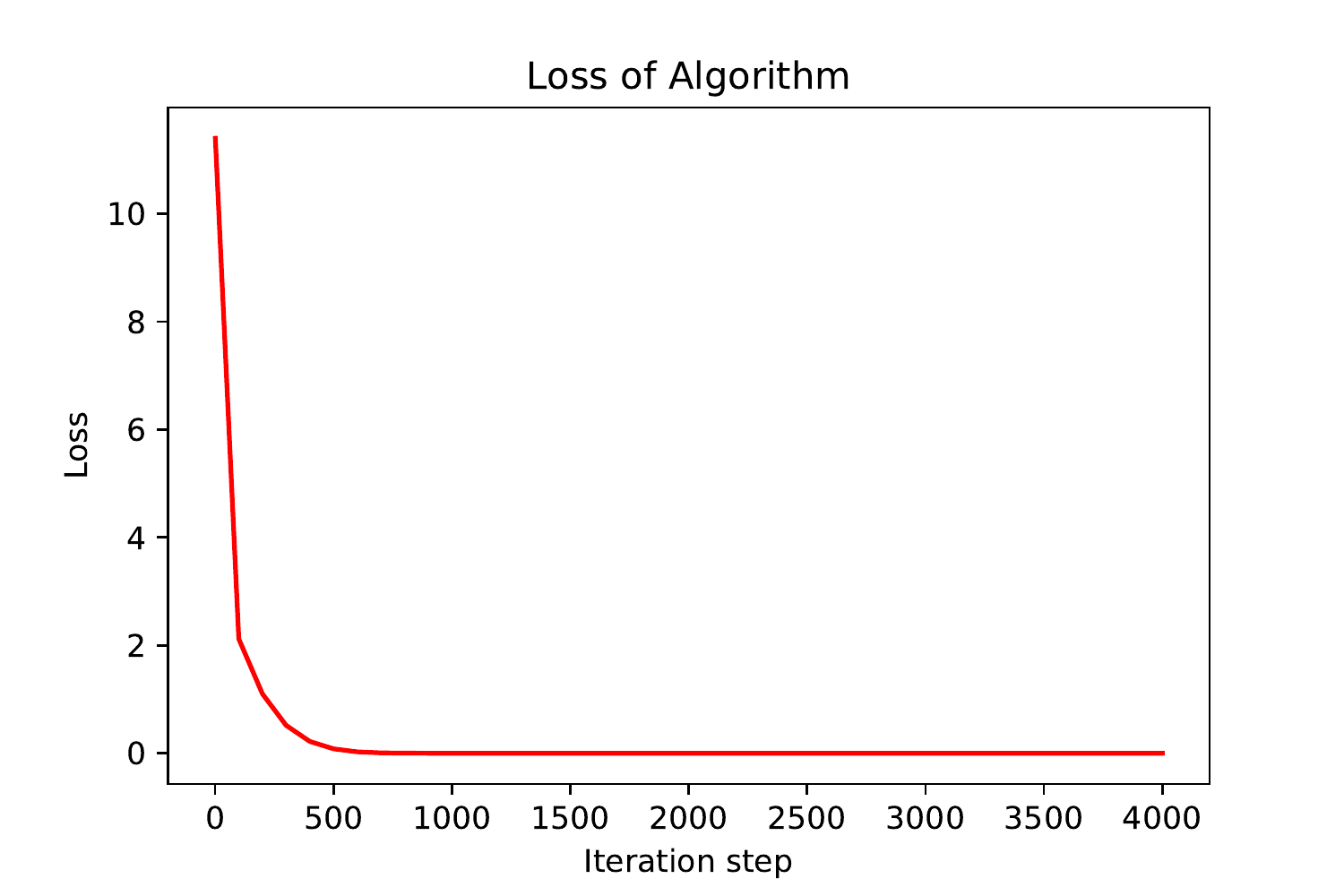}
	\includegraphics[scale=0.45]{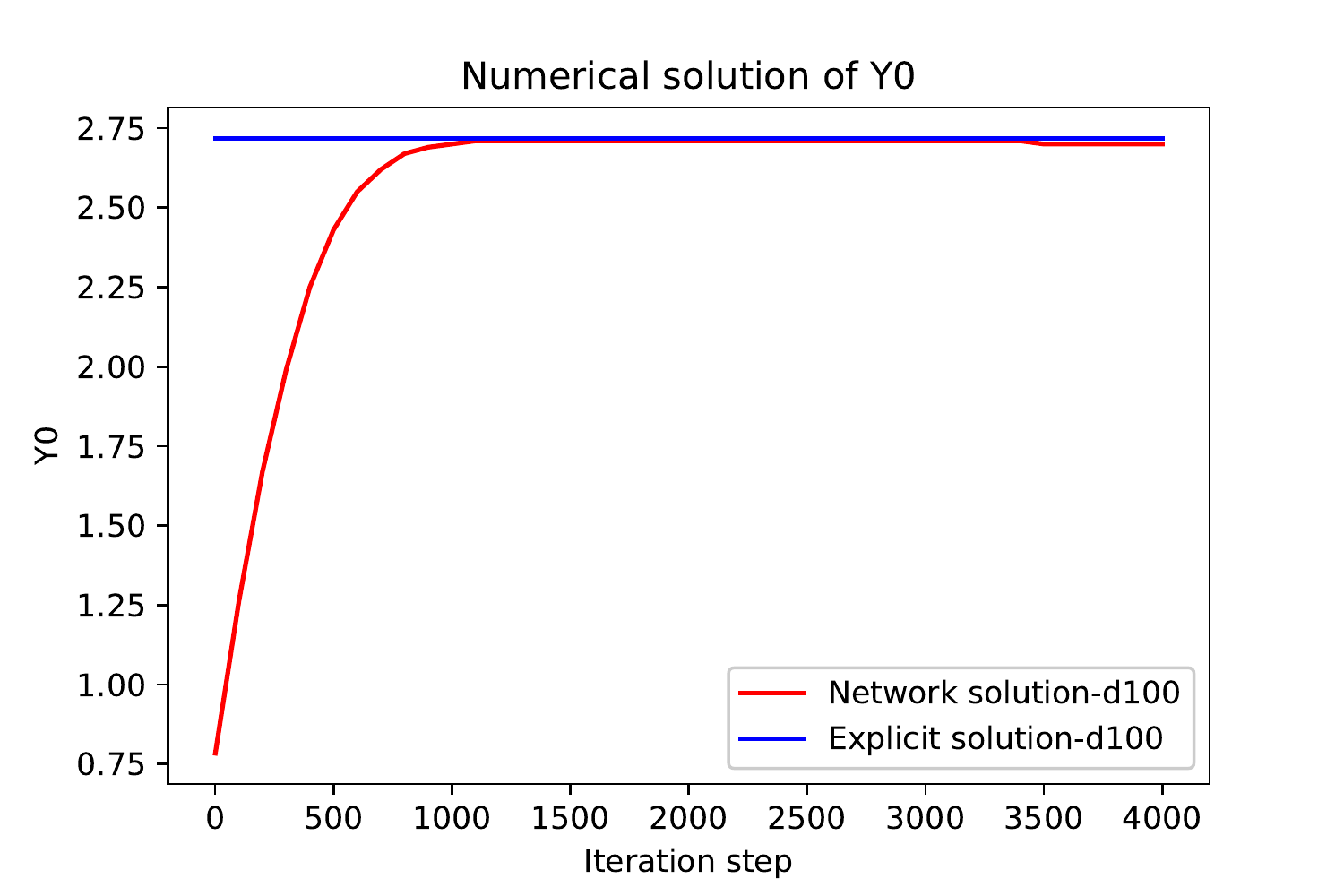}
	\caption{Case d=100. The curve of loss in the upper figure shows that after 4000 steps, the loss value is $ 1.53\times10^{-5} $. The red and blue curves in the lower figure represent respectively the network solution and the explicit solution.}
	\label{fig:ex_01_d100}
\end{figure}

We show the results with different initial values for $ d=100 $ in Table \ref{tab:ExplicitEx}. The neural network demonstrates satisfactory results. For the case of $x=0.5$ and $ x=1.0 $, we compare the three algorithms with the same stopping condition as mentioned in Example 3. The comparison results of the three algorithms are shown in Table \ref{tab:exp=0.5} for $x=0.5$ and Table \ref{tab:exp=1.0} for $ x=1.0 $ respectively.

\begin{table}[H]
	\centering
	\caption{ Numerical results for $ d=100 $ with different initial values}
	\label{tab:ExplicitEx}
	\begin{tabular}{|c|c|c|c|c|c|}
		\hline
		&$ x = 0.0 $&$ x = 0.1 $&$ x = 0.2 $&$ x = 0.5 $&$ x = 1.0 $\\
		\hline
		Explicit solution &1.00000&1.10517&1.22140&1.64872&2.71828\\
		\hline
		Network solution  &0.99777&1.10390&1.22103&1.64780&2.71662\\
		\hline
		Absolute error	  &2.23E-3&1.27E-3&3.71E-4&9.21E-4&1.66E-3\\
		\hline
		Relative error	  &2.23E-3&1.15E-3&3.03E-4&5.58E-4&6.11E-4\\
		\hline
\end{tabular}
\end{table}

\begin{table}[H]
	\centering
	\caption{ Numerical results comparison of the 3 algorithms in case $x=0.5$}
	\label{tab:exp=0.5}
	\begin{tabular}{|c|c|c|c|c|c|}
		\hline
        Method&Mean of $ Y_0 $&Variance of $ Y_0 $&Relative error of $ Y_0 $&Steps&Time(s) \\
		\hline
		Alg 1	&1.6481&4.851E-07&3.93E-04&1518.2&1000.8\\
		\hline
		Alg 2	&1.6487&1.262E-09&1.94E-05&2049.0&423.6\\
		\hline
		Alg 3	&1.6478&9.335E-08&5.57E-04&1087.8&1957.8\\
		\hline
	\end{tabular}
\end{table}

\begin{table}[H]
	\centering
	\caption{Numerical results comparison of the 3 algorithms in case $x=1.0$}
	\label{tab:exp=1.0}
	\begin{tabular}{|c|c|c|c|c|c|}
		\hline
		Method&Mean of $ Y_0 $&Variance of $ Y_0 $&Relative error of $ Y_0 $&Steps&Time(s) \\
		\hline
		Alg 1	&2.7154&3.562E-07&1.07E-03&3516.8&2205.3\\
		\hline
        Alg 2	&2.7182&4.456E-10&1.21E-05&2557.7&503.4\\
        \hline
		Alg 3	&2.7164&1.282E-07&6.76E-04&1589.2&2768.4\\
        \hline
	\end{tabular}
\end{table}

From the running results of Example 3 and 4, we can get the following phenomenon. Algorithm 3 needs least number of iteration steps to get a stable convergence rate, but takes the longest time for a given step. Algorithm 2 computes as fast or faster than Algorithm 1, probably because it has fewer network parameters. In terms of accuracy, the three algorithms show different performance results for different problems. For example, Algorithm 2 has the best variance in Example 4, but for Example 3 it can not meet the requirement of stable convergence.

\section{Conclusions}

In this paper, based on different kinds of feedback controls, we propose three algorithms for solving high-dimensional FBSDEs and construct corresponding neural networks. From the numerical results, all the three algorithms perform well for solving FBSDEs, and the relative error are less than 1\%. Algorithm 3 takes only a few steps to achieve convergence results, but each iteration may take more time. Although Algorithms 1 and 2 are computationally fast, they may require more steps to converge.

\bibliographystyle{unsrt}
\bibliography{ref}

\begin{thebibliography}{10}

\bibitem{pengBSDE1990}
Etienne Pardoux and Shige Peng.
\newblock Adapted solution of a backward stochastic differential equation.
\newblock {\em Systems and Control Letters}, 14(1):55--61, 1990.

\bibitem{Feynman_for_FBSDE}
Etienne Pardoux and Shanjian Tang.
\newblock Forward-backward stochastic differential equations and quasilinear
  parabolic pdes.
\newblock {\em Probability Theory and Related Fields}, 114(2):123--150, 1999.

\bibitem{Yong2007Forward}
Jin Ma and Jiongmin Yong.
\newblock {\em Forward-Backward stochastic differential equations and their
  applications}.
\newblock Springer, 2007.

\bibitem{Zhao2016Multistep}
Yu~Fu, Weidong Zhao, and Tao Zhou.
\newblock Multistep schemes for forward backward stochastic differential
  equations with jumps.
\newblock {\em Journal of Scientific Computing}, 69(2):1--22, 2016.

\bibitem{time_discretize_FBSDE_Zhang}
Christian Bender and Jianfeng Zhang.
\newblock Time discretization and markovian iteration for coupled fbsdes.
\newblock {\em Annals of Applied Probability}, 18(1):143--177, 2008.

\bibitem{BengioDL}
Ian Goodfellow, Yoshua Bengio, and Aaron Courville.
\newblock {\em Deep Learning}.
\newblock MIT Press, 2016.
\newblock \url{http://www.deeplearningbook.org}.

\bibitem{WeinanDLforBSDE}
Weinan E, Jiequn Han, and Arnulf Jentzen.
\newblock Deep learning-based numerical methods for high-dimensional parabolic
  partial differential equations and backward stochastic differential
  equations.
\newblock {\em Communications in Mathematics and Statistics}, 5(4):349--380,
  2017.

\bibitem{deeplearning_FBSDE}
Jiequn Han and Jihao Long.
\newblock Convergence of the deep bsde method for coupled fbsdes.
\newblock {\em arXiv:1811.01165}, 2018.

\bibitem{Peng1991Probabilistic}
Shige Peng.
\newblock Probabilistic interpretation for systems of quasilinear parabolic
  partial differential equations.
\newblock {\em Stochastics and Stochastic Reports}, 37(1):61--74, 1991.

\bibitem{Peng1999Fully}
Shige Peng and Zhen Wu.
\newblock Fully coupled forward-backward stochastic differential equations and
  applications to optimal control.
\newblock {\em Siam Journal on Control and Optimization}, 37(3):825--843, 1999.

\bibitem{Hu1995Solution}
Ying Hu and Shige Peng.
\newblock Solution of forward-backward stochastic differential equations.
\newblock {\em Probability Theory and Related Fields}, 103(2):273--283, 1995.

\bibitem{oksendal1985stochastic}
Bernt Oksendal.
\newblock Stochastic differential equations.
\newblock {\em The Mathematical Gazette}, 77(480):65--84, 1985.

\bibitem{approximation1989cybenko}
George Cybenko.
\newblock Approximation by superpositions of a sigmoidal function.
\newblock {\em Mathematics of Control, Signals, and Systems}, 2(4):303--314,
  1989.

\bibitem{Numerical1992Kloeden}
Peter~E. Kloeden and Eckhard Platen.
\newblock {\em Numerical Solution of Stochastic Differential Equations}.
\newblock Springer, 1992.

\bibitem{explictresultsimulation}
Jean-Francois Chassagneux.
\newblock Linear multi-step schemes for bsdes.
\newblock {\em SIAM Journal on Numerical Analysis}, 52(6):2815--2836, 2014.

\bibitem{Yu2016Efficient}
Yu~Fu, Weidong Zhao, and Tao Zhou.
\newblock Efficient spectral sparse grid approximations for solving
  multi-dimensional forward backward sdes.
\newblock {\em American Institute of Mathematical Sciences}, 22(9):3439--3458,
  2017.

\end{thebibliography}

\end{document}